\documentclass[english]{article}
\usepackage[margin=2.5cm]{geometry}
\usepackage[hyphens]{url}
\usepackage{hyperref}
\usepackage[hyphenbreaks]{breakurl}
\usepackage{amsthm,amsfonts,amssymb,amsmath}
\usepackage{graphicx}
\usepackage{bbm}

\newtheorem{theorem}{Theorem}[section]
\newtheorem{proposition}[theorem]{Proposition}
\newtheorem{lemma}[theorem]{Lemma}

\newtheorem{claim}[theorem]{Claim}

\newtheorem{definition}[theorem]{Definition}
\newtheorem{problem}[theorem]{Problem}
\newtheorem{setting}[theorem]{Setting}

\theoremstyle{definition}

\numberwithin{equation}{section}
\numberwithin{figure}{section}
\allowdisplaybreaks

\usepackage{babel}
\usepackage[T1]{fontenc}
\usepackage[utf8]{inputenc}

\newcommand{\su}{\subseteq}

\newcommand{\spn}{\operatorname{span}}

\begin{document}
\title{Sharp bounds for rainbow matchings in hypergraphs}
\author{Cosmin Pohoata\thanks{Department of Mathematics at Emory University, Emory University, Atlanta, GA. Email: \href{cosmin.pohoata@emory.edu}{\nolinkurl{cosmin.pohoata@emory.edu}}. Research supported by NSF Award DMS-2246659.}
\and Lisa Sauermann\thanks{Institute for Applied Mathematics, University of Bonn, Germany.
Email: \href{sauermann@iam.uni-bonn.de}{\nolinkurl{sauermann@iam.uni-bonn.de}}. Research supported by the
DFG Heisenberg Program.}
\and  Dmitrii Zakharov\thanks{Department of Mathematics, Massachusetts Institute of Technology, Cambridge, MA.
Email: \href{zakhdm@mit.edu}{\nolinkurl{zakhdm@mit.edu}}.}}
\date{}
\maketitle

\begin{abstract}
Suppose we are given matchings $M_1,....,M_N$ of size $t$ in some $r$-uniform hypergraph, and let us think of each matching having a different color. How large does $N$ need to be (in terms of $t$ and $r$) such that we can always find a rainbow matching of size $t$? This problem was first introduced by Aharoni and Berger, and has since been studied by several different authors. For example, Alon discovered an intriguing connection with the Erd\H{o}s--Ginzburg--Ziv problem from additive combinatorics, which implies certain lower bounds for $N$.

For any fixed uniformity $r \ge 3$, we answer this problem up to constant factors depending on $r$, showing that the answer is on the order of $t^{r}$. Furthermore, for any fixed $t$ and large $r$, we determine the answer up to lower order factors. We also prove analogous results in the setting where the underlying hypergraph is assumed to be $r$-partite. Our results settle questions of Alon and of Glebov--Sudakov--Szab\'{o}.
\end{abstract}

\section{Introduction}

Motivated by classical questions about transversals in Latin Squares such as the famous Ryser--Brualdi--Stein Conjecture, there has been a lot of work on finding rainbow matchings in properly edge-colored graphs (see e.g.\ \cite{chakraborti-loh, keevash-et-al} and the references therein). In a properly edge-colored graph every color class is a matching, so these questions amount to finding a rainbow matching among a collection of matchings of different colors. Similar questions have also been studied in the setting of hypergraphs, which will be the focus of this paper.

A matching in an $r$-uniform hypergraph is a collection of pairwise disjoint edges (and the size of the matching is the number of edges it consists of). Given matchings $M_1,\dots,M_N$ in some $r$-uniform hypergraph, where we think of each matching as colored in a different color, a \emph{rainbow matching} is a matching consisting of edges $e_1\in M_{i_1},\dots,e_\ell\in M_{i_\ell}$ with distinct indices $i_1,\dots,i_\ell\in \{1,\dots,N\}$ (in other words, a matching consisting of edges with distinct colors).

The following problem goes back to Aharoni and Berger \cite{aharoni-berger}.

\begin{problem}\label{problem-F}
Let $M_1,\dots,M_N$ be matchings of size $t$ in some $r$-uniform hypergraph. How large does $N$ need to be (in terms of $t$ and $r$), such that it is always possible to find a rainbow matching of size $t$? In other words, how large does $N$ need to be such that it is always possible to find distinct indices $i_1,\dots,i_t\in \{1,\dots,N\}$ and pairwise disjoint  edges $e_1\in M_{i_1},\dots,e_t\in M_{i_t}$?
\end{problem}

We remark that there may be edges belonging to more than one of the matchings $M_1,\dots,M_N$, in which case we can think of such edges having more than one color. To form a rainbow matching, one may choose which color to use these edges with.

Equivalently to Problem \ref{problem-F}, one may ask about the maximum possible number of matchings of size $t$ in some $r$-uniform hypergraph without a rainbow matching of size $t$. Denoting this maximum possible number by $F(r,t)$, the answer to Problem \ref{problem-F} is precisely $N=F(r,t)+1$.

It is also natural to ask about Problem \ref{problem-F} with the additional restriction that the underlying $r$-uniform hypergraph is $r$-partite. This was in fact the original version of the problem proposed by Aharoni and Berger in 
\cite{aharoni-berger}. Let $f(r,t)$ be the maximum possible number of matchings of size $t$ in some $r$-partite $r$-uniform hypergraph without a rainbow matching of size $t$. We clearly have the inequality $f(r,t)\le F(r,t)$.

In the case of uniformity $r=2$, i.e.\ in the case of graphs, Aharoni and Berger \cite{aharoni-berger} proved that $f(2,t)=2t-2$ (relying on previous ideas of Drisko \cite{drisko}). For $F(2,t)$ the best known bounds are $2t-2=f(2,t)\le F(2,t)\le 3t-4$, where the upper bound is due to Aharoni--Briggs--Kim--Kim \cite{abkk} (building upon work of Aharoni--Berger--Chudnovsky--Howard--Seymour \cite{abchs} showing a slightly weaker bound). In general, Glebov--Sudakov--Szab\'{o} \cite{glebov-sudakov-szabo} conjectured that $f(r,t)$ is upper-bounded by a linear function of $t$ for any fixed $r\ge 2$ (or stated more formally, that for any fixed $r\ge 2$ there is a constant $c(r)$ such that $f(r,t)\le c(r)\cdot t$ holds for all $t$). Alon \cite{alon-rainbow} had also already asked in 2011 whether this is true, based on an intriguing connection with the Erd\H{o}s--Ginzburg--Ziv problem \cite{egz} from additive combinatorics. For the fractional version of Problem \ref{problem-F} (where $M_1,\dots,M_N$ are fractional matchings and one is looking for a rainbow fractional matching), such a bound was recently proved by Aharoni, Holzman and Jiang \cite{aharoni-holzman-jiang}, using tools from topology. 

Nevertheless, it turns out that the conjecture of Glebov--Sudakov--Szab\'{o} \cite{glebov-sudakov-szabo} is actually false: In this paper, we show that for any fixed uniformity $r\ge 3$, the functions $f(r,t)$ and $F(r,t)$ are in fact on the order of $t^r$ (up to constant factors depending on $r$). 
\begin{theorem}\label{thm-fixed-r}
For any fixed uniformity $r\geq 3$, there exist positive constants $c_r$ and $C_r$ such that 
\[c_rt^r\le f(r,t)\le F(r,t)\le C_rt^r\]
holds for all $t\ge 2$.
\end{theorem}

Note that for fixed uniformity $r\ge 3$, this theorem determines $f(r,t)$ and $F(r,t)$ up to constant factors. In other words, Theorem \ref{thm-fixed-r} essentially (up to constant factors) solves Problem \ref{problem-F} (as well as its $r$-partite analogue) in the setting of fixed uniformity $r\ge 3$. 

Theorem \ref{thm-fixed-r} shows in particular that the behavior of $f(r,t)$ is very different from the behavior of Erd\H{o}s--Ginzburg--Ziv constants. Indeed, Alon--Dubiner \cite{alon-dubiner} proved that the  Erd\H{o}s--Ginzburg--Ziv constant $\mathfrak{s}(\mathbb{Z}_{t}^{r})$ is linear in $t$ for any fixed $r$ (and the linearity factor was recently improved by the third author \cite{dima-egz}). Alon \cite{alon-rainbow} established that $\mathfrak{s}(\mathbb{Z}_{t}^{r-1})\le f(r,t)$, which in part lead to the conjecture that the function $f(r,t)$ as the combinatorial counterpart to the Erd\H{o}s--Ginzburg--Ziv constant should also be linear in $t$ for fixed uniformity $r$.

The best previous lower bounds for both $f(r,t)$ and $F(r,t)$ were linear in $t$. The best previous upper bound for both $f(r,t)$ and $F(r,t)$ was $(t-1)\binom{tr}{r}$ due to Munh\'{a} Correia, Sudakov, and Tomon \cite{correia-sudakov-tomon}, stated in the following theorem. This upper bound is on the order of $t^{r+1}$ for fixed $r$, so Theorem \ref{thm-fixed-r} improves both the known lower and upper bounds in this regime.

\begin{theorem}[{\cite{correia-sudakov-tomon}}]
\label{thm-correia-sudakov-tomon}
For any $t\ge 2$ and $r\ge 2$, we have 
$$f(r,t)\le F(r,t)\le (t-1)\binom{tr}{r}.$$
\end{theorem}

This upper bound due to Munh\'{a} Correia, Sudakov, and Tomon \cite{correia-sudakov-tomon} improved upon previous upper bounds of Alon \cite{alon-rainbow} and Glebov--Sudakov--Szab\'{o} \cite{glebov-sudakov-szabo}. Their proof of $F(r,t)\le (t-1)\binom{tr}{r}$ is of a linear algebraic nature, using exterior power algebras. For such an approach $(t-1)\binom{tr}{r}$ is an important barrier, because this bound is tight for the natural linear-algebraic abstraction of the problem (see the discussion in Section \ref{subsect-linear-algebra}).

Our proof of the upper bound in Theorem \ref{thm-fixed-r} is purely combinatorial, and is motivated by arguments that first appeared in the context of the famous sunflower problem \cite{alweiss-lovett-wu-zhang,lovett-solomon-zhang} and were later used in the resolution of the fractional Kahn--Kalai conjecture by Frankston--Kahn--Narayanan--Park \cite{frankston-kahn-narayanan-park}. For the lower bound in Theorem \ref{thm-fixed-r} we give an explicit construction.

It is also interesting to study the functions $f(r,t)$ and $F(r,t)$ in the opposite regime, where $t\ge 2$ is fixed and $r$ is large. For example, Alon \cite{alon-rainbow} explicitly asked about studying the function $f(r,t)$ in the case of $t=3$ and large $r$, i.e.\ about understanding the growth behavior of $f(r,3)$ as a function of $r$. Furthermore, for $t=2$, Problem \ref{problem-F} is closely related to Bollob\'as' celebrated ``Two Families Theorem''. So studying this problem for fixed $t\geq 3$, can be viewed as a natural extension of Bollob\'as'  ``Two Families Theorem'' to a ``$t$ Families Theorem''. 

For fixed $t\ge 2$ and large $r$, the upper bound $(t-1)\binom{tr}{r}$ for $F(r,t)$ due to Munh\'{a} Correia, Sudakov, and Tomon is (up to constant factors depending on $t$) on the order of $(t^t/(t-1)^{t-1})^r/\sqrt{r}$, i.e.\ it is exponential in $r$ with base $t^t/(t-1)^{t-1}$ (which, for relatively large $t$, is roughly $et$). We prove that this upper bound is tight up to sub-exponential factors in $r$.

\begin{theorem}\label{thm-fixed-t-F}
For any fixed $t\ge 2$, and any large $r$, we have
\[F(r,t)\ge \left(\frac{t^t}{(t-1)^{t-1}}\right)^{r-O(\sqrt{r})}.\]
\end{theorem}
Here, the implicit constant in the $O$-notation may depend on $t$.

For $f(r,t)$, i.e.\ in the setting of $r$-partite graphs, an easy adaptation of the proof of Munh\'{a} Correia, Sudakov, and Tomon gives the upper bound $f(r,t)\le (t-1)\cdot t^r$. Again, we prove that for fixed $t\ge 2$ and large $r$, this is tight up to sub-exponential factors in $r$.

\begin{theorem}\label{thm-fixed-t-f}
For any fixed $t\ge 2$, and any $r\ge 2$, we have
\[t^{r-O(\sqrt{r})}\le f(r,t)\le (t-1)\cdot t^r.\]
\end{theorem}

Again, the implicit constant in the $O$-notation may depend on $t$.

Theorem \ref{thm-fixed-t-f} in particular answers a question of Glebov--Sudakov--Szab\'{o} \cite{glebov-sudakov-szabo}, asking whether $f(r,t)$ can be upper-bounded by a function of the form $\alpha_t\beta^r$ (where $\alpha_t$ may depend on $t$, but $\beta$ is an absolute constant). In fact, the best previous lower bounds for both $f(r,t)$ and $F(r,t)$ were of this form with $\beta=e\approx 2.71$ (to be precise, the best previous lower bound, due to Alon \cite{alon-rainbow}, was $(e-o(1))^{r-1}$ where the $o(1)$-term converges to zero as $t$ grows). However, Theorem \ref{thm-fixed-t-f} shows that one cannot get such an upper bound with an absolute constant $\beta$.

The lower bounds in Theorems \ref{thm-fixed-t-F} and \ref{thm-fixed-t-f} can be proved via a probabilistic argument with some ingredients from additive combinatorics. This argument is similar to the proofs of the lower bounds for the so-called multi-colored sum-free theorem in \cite{kleinberg-sawin-speyer} and \cite{lovasz-sauermann}. In fact, to prove Theorem \ref{thm-fixed-t-F}, one can use \cite[Proposition 3.1]{lovasz-sauermann} as a black-box, but for the setting of Theorem \ref{thm-fixed-t-f} one needs to adapt the strategy and use different arguments to control linear dependencies that affect the relevant probabilistic bounds.

Note that Theorems \ref{thm-fixed-t-F} and \ref{thm-fixed-t-f} (together with the upper bound $F(r,t)\le (t-1)\binom{tr}{r}$ due to Munh\'{a} Correia, Sudakov, and Tomon) determine the functions $f(r,t)$ and $F(r,t)$ for fixed $t\ge 2$ up to lower-order terms (i.e.\ up to sub-exponential terms in $r$). 

\textit{Organization.} We first prove Theorem \ref{thm-fixed-r} about fixed uniformity $r$, showing the upper bound in Section \ref{sect-upper-bound-fixed-r} and the lower lower bound in Section \ref{sect-lower-bound-fixed-r}.  Afterwards, we prove Theorems \ref{thm-fixed-t-F} and \ref{thm-fixed-t-f} about fixed matching size $t$ in Section \ref{sect-fixed-t}.

\textit{Acknowledgements.} The authors would like to thank the anonymous referee for their careful reading of the paper and many helpful suggestions. The majority of this work was completed while the first author was affialated with the Institute for Advanced Study and supported by NSF Award DMS-1926686, and while the second author was affiliated with the Massachusetts Institute of Technology and supported by NSF Award DMS-2100157 and a Sloan Research Fellowship.

\section{Upper bound for fixed uniformity}
\label{sect-upper-bound-fixed-r}

In this section, we prove the upper bound in Theorem \ref{thm-fixed-r}. More precisely, we prove the following:

\begin{theorem}\label{thm-fixed-r-with-bound}
Let $r\geq 2$ and $t\ge 2$ be integers. Let $M_1,\dots,M_N$ be matchings of size $t$ in some $r$-uniform hypergraph, with $N\ge (tr+t)^r$. Then there exists a rainbow matching of size $t$ (in other words, there exist distinct indices $i_1,\dots,i_t\in \{1,\dots,N\}$ and pairwise disjoint edges $e_1\in M_{i_1},\dots,e_t\in M_{i_t}$).
\end{theorem}

This theorem shows that $F(r,t)\le (tr+t)^r$ for $r\geq 2$ and $t\ge 2$, and therefore implies the upper bound in Theorem \ref{thm-fixed-r} with $C_r=(r+1)^r$.

\begin{proof}
Let $V$ be the vertex set of the underlying $r$-uniform hypergraph. Now, let $\mathcal{F}=M_1\cup \dots\cup M_N$ be the family of all edges appearing in at least one of the matchings $M_1,\dots,M_N$. Note that $\mathcal{F}$ is a family of subsets of $V$, each of size $r$.

For a subset $S\su V$ and sub-family $\mathcal{G}\su \mathcal{F}$, let $\mathcal{G}(S)=\{X\in \mathcal{G}\mid S\su X\}$ denote the family of all edges in $\mathcal{G}$ that contain $S$. Then $|\mathcal{G}(S)|$ is the number of edges in $\mathcal{G}$ containing $S$. Note that $\mathcal{G}(\emptyset)=\mathcal{G}$ and in particular $|\mathcal{G}(\emptyset)|=|\mathcal{G}|$. Furthermore note that $|\mathcal{G}(S)|=0$ whenever $|S|>r$ and $|\mathcal{G}(S)|\le 1$ whenever $|S|=r$.

Let us now define a sequence $\mathcal{F}=\mathcal{F}_1\supset \mathcal{F}_1\supset \dots \supset \mathcal{F}_\ell$ as follows. Define $\mathcal{F}_1=\mathcal{F}$. Now for some positive integer $k$, assume that we have already defined $\mathcal{F}_k$. If $|\mathcal{F}_k|\le (tr+t)^r$, we stop the process at this point (setting $\ell=k$). Otherwise, if $|\mathcal{F}_k|> (tr+t)^r$, let us consider a maximal subset $S_k\su V$ with the property that
\begin{equation}\label{eq-spread}
|\mathcal{F}_k(S_k)|\ge \frac{|\mathcal{F}_k|}{(tr+t)^{|S_k|}},
\end{equation}
and let us then define $\mathcal{F}_{k+1}=\mathcal{F}_{k}\setminus \mathcal{F}_k(S_k)$ (i.e.\ $\mathcal{F}_{k+1}$ is obtained from $\mathcal{F}_k$ by removing all edges containing $S_k$).

Let us remark that at every step of this process, there is indeed a subset $S_k\su V$ satisfying (\ref{eq-spread}), since $S_k=\emptyset$ always has this property (and so, among all subsets $S_k\su V$ satisfying (\ref{eq-spread}) we can choose a maximal one since $V$ is a finite set). 

Also note that by (\ref{eq-spread}) at every step $\mathcal{F}_k(S_k)$ is non-empty and so $\mathcal{F}_{k+1}=\mathcal{F}_{k}\setminus \mathcal{F}_k(S_k)$ is a proper subset of $\mathcal{F}_k$. In particular, since $\mathcal{F}_1=\mathcal{F}$ is a finite set, this means that the sequence $\mathcal{F}=\mathcal{F}_1\supset \mathcal{F}_1\supset \dots \supset \mathcal{F}_\ell$ needs to indeed terminate eventually (i.e.\ the above process must indeed stop at some point after some finite number of steps). By the definition of the process, we have $|\mathcal{F}_\ell|\le (tr+t)^r$ for the last family in the sequence (and we have $|\mathcal{F}_k|>(tr+t)^r$ for $k=1,\dots,\ell-1$).

Recall that whenever $|\mathcal{F}_k|\le (tr+t)^r$, the process terminates and we do not need to choose a subset $S_k\su V$ (and we do not define another family $\mathcal{F}_{k+1}$). We claim that whenever $S_k$ is defined (i.e.\ whenever $|\mathcal{F}_k|> (tr+t)^r$), we have $|S_k|\le r-1$. Indeed, if $|S_k|>r$, then $|\mathcal{F}_k(S_k)|=0$ while the right-hand side of (\ref{eq-spread}) is positive, so the inequality in (\ref{eq-spread}) cannot be satisfied. If $|S_k|=r$, then $|\mathcal{F}_k(S_k)|\le 1$ while  the right-hand side of (\ref{eq-spread}) is larger than $1$ (since $|\mathcal{F}_k|> (tr+t)^r$),  so again (\ref{eq-spread}) cannot be satisfied. This shows that indeed all the sets $S_k$ for $k=1,\dots,\ell-1$ occurring in this process satisfy $|S_k|\le r-1$.

Recall that $|\mathcal{F}_\ell|\le (tr+t)^r\le N$. Let us imagine that every edge in $\mathcal{F}_\ell$ is given one dollar (then the total amount of money is $|\mathcal{F}_\ell|\le N$ dollars). Every edge in $\mathcal{F}_\ell\su \mathcal{F}$ is contained in at least one of the matchings $M_1,\dots,M_N$. Now, imagine that every edge $e\in \mathcal{F}_\ell$ gives its one dollar in equal shares to all the indices $i\in \{1,\dots,N\}$ which satisfy $e\in M_i$ (if there are $z$ such indices, then $e$ gives each of them $1/z$ dollars). In other words, if we imagine the matchings $M_1,\dots,M_N$ as colors, each edge gives its one dollar in equal shares to all of its colors.

Since the total amount of money is at most $N$ dollars, there must be an index $j\in \{1,\dots,N\}$ receiving at most one dollar. Let us fix such an index $j$.

Recall that $M_j$ is a matching of size $t$. Let $m=\left|M_j\cap \mathcal{F}_\ell\right|$ be the number of edges of $M_j$ that are part of the family $\mathcal{F}_\ell$ (these are the edges contributing money to the index $j$ in the process described above), then $0\le m\le t$. Let us denote the $t$ edges in $M_j$ by $e^*_1,\dots,e^*_t$ in such an order that we have $e^*_1,\dots,e^*_m\in \mathcal{F}_\ell$ and $e^*_{m+1},\dots,e^*_t\not\in \mathcal{F}_\ell$.

Our goal is now to find a rainbow matching of size $t$ by adapting the matching $M_j$ in a suitable way. First, we use the following claim for the first $m$ edges $e^*_1,\dots,e^*_m$ of $M_j$.

\begin{claim}\label{claim-edges-from-F-ell}
There exists $m$ distinct indices $i_1,\dots,i_m\in \{1,\dots,N\}$ such that $e_1^*\in M_{i_1},\dots, e_m^*\in M_{i_m}$.
\end{claim}
\begin{proof}
Let us consider the bipartite graph with vertices labeled by $e_1^*,\dots, e_m^*$ on the left side and vertices labeled by $1,\dots,N$ on the right side, where we draw an edge between vertex $e_h^*$ on the left and vertex $i$ on the right (for $1\le h\le m$ and $1\le i\le N$) if and only if $e_h^*\in M_i$. Note that vertex $j$ on the right side is adjacent to all vertices on the left side (since $e^*_1,\dots,e^*_m\in M_j$ by definition). The claim asserts that this bipartite graph has a matching of size $m$ (i.e.\ a matching that covers all the vertices $e_1^*,\dots, e_m^*$ on the left side).

In order to obtain such a matching in this auxiliary bipartite graph, by Hall's marriage theorem it suffices to check that for any $1\le a\le m$, and any subset of $a$ vertices on the left side, the union of their neighborhoods on the right side has size at least $a$. So suppose there exist $a$ vertices on the left side, such that the union of their neighborhoods on the right side has size at most $a-1$.

Let us consider the $a$ dollars that these $a$ vertices on the left side obtained (recall that $e^*_1,\dots,e^*_m\in \mathcal{F}_\ell$, so every vertex on the left side of this bipartite graph obtained a dollar). Each of these vertices gives its dollar in equal shares to its neighbors $i\in \{1,\dots,N\}$ on the right side of this bipartite graph (since for every vertex $e_h^*$ on the left side its neighbors on the right side are precisely those indices $i\in \{1,\dots,N\}$ with $e_h^*\in M_i$). By our assumption, at most $a-1$ vertices are receiving some share of these $a$ dollars, so one vertex $j'\in \{1,\dots,N\}$ on the left side must receive a share of more than one dollar among these $a$ dollars. On the other hand, all vertices on the left side are adjacent to vertex $j$ on the right side, so each of the $a$ vertices we considered on the left side gives at least as much money to index $j$ as to index $j'$. Thus, $j$ must also receive a share of more than one dollar among the $a$ dollars we considered. However, by our choice of $j$, the total amount of money that $j$ receives is at most one dollar, so this is a contradiction.
\end{proof}

Claim \ref{claim-edges-from-F-ell} states that we can assign distinct colors to the first $m$ edges $e^*_1,\dots,e^*_m$ of $M_j$ (i.e.\ to the edges of $m_j$ that are contained in $\mathcal{F}_\ell$). We will now use the following claim to modify the remaining edges $e^*_{m+1},\dots,e^*_t$ in such a way that we obtain a rainbow matching.

\begin{claim}\label{claim-edges-outside-F-ell}
For each $h \in \{m,\ldots,t\}$, there exist distinct indices $i_1,\dots,i_h\in \{1,\dots,N\}$ and edges $e_1\in M_{i_1},\dots, e_h\in M_{i_h}$ such that the edges $e_1,\dots,e_h,e^*_{h+1},\dots,e^*_t$ are pairwise disjoint (in other words, the edges $e_1,\dots,e_h,e^*_{h+1},\dots,e^*_t$ form a matching of size $t$).
\end{claim}
Note that for $h=t$, this claim states precisely that there is a rainbow matching of size $t$ among our original matchings $M_1,\dots,M_t$. Thus, showing this claim concludes the proof of Theorem \ref{thm-fixed-r-with-bound}. 

\begin{proof}
We prove the claim by induction on $h$. For $h=m$, the statement follows from Claim \ref{claim-edges-from-F-ell}. Indeed, we can take distinct indices $i_1,\dots,i_m\in \{1,\dots,N\}$ as in Claim \ref{claim-edges-from-F-ell} and define $e_a=e^*_a$ for $a=1,\dots,m$. Then we have $e_1\in M_{i_1},\dots, e_m\in M_{i_m}$ and the edges $e_1,\dots,e_m,e^*_{m+1},\dots,e^*_t$ are pairwise disjoint (since they are precisely the edges $e^*_1,\dots,e^*_t$ of the matching $M_j$).

So let us now assume that $m+1\le h\le t$, and that we already proved Claim \ref{claim-edges-outside-F-ell} for $h-1$. This means that there exist distinct indices $i_1,\dots,i_{h-1}\in \{1,\dots,N\}$ and edges $e_1\in M_{i_1},\dots, e_{h-1}\in M_{i_{h-1}}$ such that the edges $e_1,\dots,e_{h-1},e^*_{h},\dots,e^*_t$ are pairwise disjoint.

Recall that $e^*_h\not\in \mathcal{F}_\ell$ (since $h\ge m+1$). As $e^*_h\in M_j\su \mathcal{F}=\mathcal{F}_1$, there must be an index $k\in \{1,\dots,\ell-1\}$ such that $e_h^*\in \mathcal{F}_k\setminus \mathcal{F}_{k+1}$. Recalling that we defined $\mathcal{F}_{k+1}=\mathcal{F}_k\setminus \mathcal{F}_k(S_k)$, this means that $e_h^*\in \mathcal{F}_k(S_k)$, so $S_k\su e_h^*$. In particular, this means that $S_k$ is disjoint from $e_1,\dots,e_{h-1}$ and $e^*_{h+1},\dots,e^*_t$.

Recall from (\ref{eq-spread}) that
\[|\mathcal{F}_k(S_k)|\ge \frac{|\mathcal{F}_k|}{(tr+t)^{|S_k|}}.\]
On the other hand, $S_k$ was chosen to be a maximal subset of $V$ satisfying (\ref{eq-spread}). Hence, for every vertex $v\in e_1\cup \dots\cup e_{h-1}\cup e^*_{h+1}\cup \dots \cup e^*_t$, we have 
\[|\mathcal{F}_k(S_k\cup \{v\})|< \frac{|\mathcal{F}_k|}{(tr+t)^{|S_k|+1}}.\]
Therefore the total number of edges in $\mathcal{F}_k(S_k)$ that contain at least one vertex in the union $e_1\cup \dots\cup e_{h-1}\cup e^*_{h+1}\cup \dots \cup e^*_t$ is at most
\[\sum_{v}|\mathcal{F}_k(S_k\cup \{v\})|< tr\cdot \frac{|\mathcal{F}_k|}{(tr+t)^{|S_k|+1}}=\frac{r}{r+1}\cdot \frac{|\mathcal{F}_k|}{(tr+t)^{|S_k|}}\]
(where the sum on the left-hand side is over all $v\in e_1\cup \dots\cup e_{h-1}\cup e^*_{h+1}\cup \dots \cup e^*_t$, noting that $|e_1\cup \dots\cup e_{h-1}\cup e^*_{h+1}\cup \dots \cup e^*_t|=(t-1)r\le tr$).
Thus, the family $\mathcal{F}_k(S_k)$ contains at least
\[|\mathcal{F}_k(S_k)|-\frac{r}{r+1}\cdot \frac{|\mathcal{F}_k|}{(tr+t)^{|S_k|}}\ge  \frac{|\mathcal{F}_k|}{(tr+t)^{|S_k|}}-\frac{r}{r+1}\cdot \frac{|\mathcal{F}_k|}{(tr+t)^{|S_k|}}=\frac{1}{r+1}\cdot \frac{|\mathcal{F}_k|}{(tr+t)^{|S_k|}}> \frac{1}{r+1}\cdot (tr+t)^{r-|S_k|}\]
edges that are disjoint from $e_1,\dots,e_{h-1}$ and $e^*_{h+1},\dots,e^*_t$ (in the last step, we used that $|\mathcal{F}_k|>(tr+t)^{r}$ since $k<\ell$). We claim  that these edges cannot all be contained in $M_{i_1}\cup \dots\cup M_{i_{h-1}}$. 

Indeed, if $S_k\neq \emptyset$, the family $\mathcal{F}_k(S_k)$ contains at most one edge from each of the matchings $M_{i_1}, \dots, M_{i_{h-1}}$ (since the edges in each matching are pairwise disjoint) and therefore at most $h-1<t$ edges in $M_{i_1}\cup \dots\cup M_{i_{h-1}}$. However, the number of edges in $\mathcal{F}_k(S_k)$ that are disjoint from $e_1,\dots,e_{h-1}$ and $e^*_{h+1},\dots,e^*_t$ is at least $(tr+t)^{r-|S_k|}/(r+1)\ge t$ (as $|S_k|\le r-1$). So there must indeed be at least one edge in $\mathcal{F}_k(S_k)$ that is disjoint from $e_1,\dots,e_{h-1}$ and $e^*_{h+1},\dots,e^*_t$ and is not contained in $M_{i_1}\cup \dots\cup M_{i_{h-1}}$.

In the other case, where $S_k= \emptyset$, the number of edges in $\mathcal{F}_k(S_k)$ that are disjoint from $e_1,\dots,e_{h-1}$ and $e^*_{h+1},\dots,e^*_t$ is at least $(tr+t)^{r-|S_k|}/(r+1)\ge (r+1)t^2$ (as $r-|S_k|=r\ge 2$). On the other hand, the number of edges in $M_{i_1}\cup \dots\cup M_{i_{h-1}}$ is at most $(h-1)t< t^2$, so again there must be at least one edge in $\mathcal{F}_k(S_k)$ that is disjoint from $e_1,\dots,e_{h-1}$ and $e^*_{h+1},\dots,e^*_t$ and is not contained in $M_{i_1}\cup \dots\cup M_{i_{h-1}}$.

In either case, let $e_h\in \mathcal{F}_k(S_k)$ be such an edge in $\mathcal{F}_k(S_k)$ that is disjoint from $e_1,\dots,e_{h-1}$ and $e^*_{h+1},\dots,e^*_t$ and not contained in $M_{i_1}\cup \dots\cup M_{i_{h-1}}$. Then, as desired, the edges $e_1,\dots,e_{h-1}, e_h, e^*_{h+1},\dots,e^*_t$ are pairwise disjoint (recall that $e_1,\dots,e_{h-1}, e^*_{h+1},\dots,e^*_t$ must be pairwise disjoint since we started with pairwise disjoint edges $e_1,\dots,e_{h-1},e^*_{h},\dots,e^*_t$). Since $e_h\in  \mathcal{F}_k(S_k)\su \mathcal{F}_k\su \mathcal{F}=M_1\cup \dots\cup M_N$ and $e_h\not\in M_{i_1}\cup \dots\cup M_{i_{h-1}}$, there must be an index $i_h\in \{1,\dots,N\}\setminus\{i_1,\dots,i_{h-1}\}$ with $e_h\in M_{i_h}$. Now $i_1,\dots,i_h$ are distinct, and we have $e_1\in M_{i_1},\dots,e_h\in M_{i_h}$ as desired. This finishes the induction step.
\end{proof}

We already saw that Claim \ref{claim-edges-outside-F-ell} finishes the proof of Theorem \ref{thm-fixed-r-with-bound}.\end{proof}

We remark that the condition in (\ref{eq-spread}) is motivated by the concept of  ``spread'' that first appeared in the context of the famous sunflower problem \cite{alweiss-lovett-wu-zhang, lovett-solomon-zhang}, and was also used in the proof of the fractional Kahn--Kalai conjecture by Frankston--Kahn--Narayanan--Park \cite{frankston-kahn-narayanan-park}. The idea of approximating an arbitrary set family by a union of well-spread families with small `cores' also appeared in \cite{kupavskii-zakharov} in the context of forbidden intersection problems. We also remark that the greedy process in the proof of Claim \ref{claim-edges-outside-F-ell} above is similar to an argument of Glebov--Sudakov--Sazb\'{o} in \cite{glebov-sudakov-szabo} (namely to the deduction of Theorem 2 from Lemma 7 on the second half of page 7 of \cite{glebov-sudakov-szabo}).

\section{Lower bound for fixed uniformity}
\label{sect-lower-bound-fixed-r}

After proving the upper bound in Theorem \ref{thm-fixed-r} in the previous section, here we show the lower bound.

\begin{theorem}\label{thm-construction-fixed-r}
Let $r\ge 3$ and $t\ge 2r$ be integers, and let $N=(\lfloor t/r\rfloor -1)^r$. Then there exist matchings $M_1,\dots,M_N$ of size $t$ in some $r$-partite $r$-uniform hypergraph with $tr$ vertices such that $M_1,\dots,M_N$ do not have a rainbow matching of size $t$.
\end{theorem}

This theorem shows that $f(r,t)\ge (\lfloor t/r\rfloor -1)^r$ for $t\ge 2r$ (and $r\geq 3$).  For $t\ge  3r$, this implies
\[f(r,t)\ge (\lfloor t/r\rfloor -1)^r\ge \left(\frac{t}{r}-2\right)^2\ge \left(\frac{t}{3r}\right)^r=(3r)^{-r}\cdot t^r.\]
For $2\le t\le 3r$, we trivially have  $f(r,t)\ge 1\ge (3r)^{-r}\cdot t^r$. Thus, we obtain $f(t,r)\ge (3r)^{-r}\cdot t^r$ for all $r\ge 3$ and $t\ge 2$, showing the lower bound in Theorem \ref{thm-fixed-r} with $c_r=(3r)^{-r}$.

Theorem \ref{thm-construction-fixed-r} asserts the existence of $N$ matchings of size $t$ in an ($r$-partite) $r$-uniform hypergraph with $tr$ vertices. Note that this automatically means that each of the matchings $M_1,\dots,M_N$ is a perfect matching in this hypergraph.

In the proof of Theorem \ref{thm-construction-fixed-r}, we will use the following notation. For an $r$-tuple $(x_1,x_2,\dots,x_r)$ of positive integers, define $\sigma(x_1,x_2, \dots,x_r)$ to be the $r$-tuple $(x_2,x_3,\dots,x_r,x_1)$, i.e.\ the $r$-tuple obtained from $(x_1,x_2,\dots,x_r)$ by a cyclic forward-shift. For $1\le j\le r-1$, let us furthermore define $\sigma^j(x_1,x_2,\dots,x_r)$ to be the $r$-tuple $\sigma(\sigma(\dots \sigma(x_1,x_2,\dots,x_r)\dots))$ obtained by applying $\sigma$ as an operator $j$ times.

\begin{proof}[Proof Theorem \ref{thm-construction-fixed-r}]
We will construct an $r$-partite $r$-uniform hypergraph on $tr$ vertices. The vertex set of this hypergraph is split into $r$ parts, each of size $t$. In each part, let us label the $t$ vertices by $1,\dots,t-r$ and $a_1,\dots,a_r$ (the vertices labeled by $a_1,\dots,a_r$ will play a special role in the construction of our matchings). Let $\mathcal{H}$ be the complete $r$-partite $r$-uniform hypergraph between these $r$ parts.

Now, the edges of $\mathcal{H}$ correspond to $r$-tuples in the set $(\{1,\dots,t-r\}\cup\{a_1,\dots,a_r\})^r$ (where any edge $e\in \mathcal{H}$ corresponds to the $r$-tuple obtained by first taking the label of the vertex of $e$ in the first part of the $r$-partition, then the label of the vertex of $e$ in the second part, and so on).

A matching of size $t$ in $\mathcal{H}$ corresponds to a collection of $t$ different $r$-tuples in $(\{1,\dots,t-r\}\cup\{a_1,\dots,a_r\})^r$, such that any of the $t$ symbols $1,\dots,t-r,a_1,\dots,a_r$ occurs among these $r$-tuples  exactly once in every position.

In order to construct the desired matchings, let us fix a partition $\{1,\dots,t-r\}=X_1\cup\dots\cup X_r$ such that each of the sets $X_1,\dots,X_r$ has size at least $\lfloor (t-r)/r\rfloor=\lfloor t/r\rfloor -1$. Then $|X_1|\dotsm |X_r|\ge (\lfloor t/r\rfloor -1)^r=N$.

We will now define a matching $M_{x_1,\dots,x_r}$ of size $t$ in $\mathcal{H}$ for every $r$-tuple $(x_1,\dots,x_r)\in X_1\times\dots\times X_r$. It will then suffice to show that the collection of these matchings $M_{x_1,\dots,x_r}$ for $(x_1,\dots,x_r)\in X_1\times\dots\times X_r$ does not have a rainbow matching of size $t$ (indeed, as $|X_1\times\dots\times X_r|\ge N$, then we can take $M_1,\dots,M_N$ to be any $N$ of these matchings $M_{x_1,\dots,x_r}$).

For each $(x_1,\dots,x_r)\in X_1\times\dots\times X_r$, let us define the matching $M_{x_1,\dots,x_r}$ to consist of the $t$ edges in $\mathcal{H}$ corresponding to the following $t$ different $r$-tuples in $(\{1,\dots,t-r\}\cup\{a_1,\dots,a_r\})^r$:
\begin{itemize}
\item[(i)] Take the $r$-tuple $(a_1,\dots,a_r)$.
\item[(ii)] For each $j \in \{1,\dots,r\}$, take the $r$-tuple $(a_j,\dots,a_j,x_j,a_j,\dots,a_j)$ where $x_j$ is in the $j$-th position (and all other entries are $a_j$).
\item[(iii)] For each $j \in \{1,\dots,r-1\}$, take the $r$-tuple $\sigma^j(x_1,x_2,\dots,x_r)$ (see the definition right before the start of this proof).
\item[(iv)] For each $i\in \{1,\dots,t-r\}\setminus \{x_1,\dots,x_r\}$, let us take the $r$-tuple $(i,i,\dots,i)$.
\end{itemize}
Let us now check that this indeed leads to a matching of size $t$. First, note that $x_1,\dots,x_r$ are distinct (as $X_1\,\dots,X_r$ are pairwise disjoint). Hence the number of $r$-tuples taken in (iv) is $t-2r$, whereas the numbers for (i), (ii) and (iii) are 1, $r$ and $r-1$, respectively. Thus, $M_{x_1,\dots,x_r}$ indeed consists of $t$ edges. In order to check that $M_{x_1,\dots,x_r}$ is indeed a matching, it is easy to verify that each of the symbols $1,\dots,t-r,a_1,\dots,a_r$ occurs exactly once in each of the $r$ positions among the $r$-tuples defined above. Indeed every symbol $a_h$ for $h=1,\dots,r$ occurs in the $h$-th position in (i) and in all other positions in (ii) for $j= h$;  every symbol $x_h$ for $h=1,\dots,r$ occurs in the $h$-th position in (ii) for $j=h$ and in all other positions in (iii); and every symbol $i\in\{1,\dots,t-r\}\setminus \{x_1,\dots,x_r\}$ occurs in all positions in (iv). Thus $M_{x_1,\dots,x_r}$ is indeed a matching of size $t$ for each $(x_1,\dots,x_r)\in X_1\times\dots\times X_r$.

Note that for each $h \in \{1,\dots,r\}$, in all of $r$-tuples in (iii) the elements $X_h$ can only occur outside of the $h$-th position (indeed, in $\sigma^j(x_1,x_2,\dots,x_r)$, the position of the element $x_h\in X_h$ is congruent to $h+j\not\equiv h\pmod{r}$). Therefore, any edge in any of the matchings $M_{x_1,\dots,x_r}$ that has an element of $X_h$ in the $h$-th position must be obtained from rule (ii) or (iv).

It remains to show that there does not exist a rainbow matching of size $t$ for the matchings $M_{x_1,\dots,x_r}$ for $(x_1,\dots,x_r)\in X_1\times\dots\times X_r$. So suppose $M$ is such a rainbow matching of size $t$ (i.e.\ a matching of size $t$ that can be obtained by choosing $t$ distinct matchings $M_{x_1,\dots,x_r}$ in our collection and selecting one edge from each of them). Since the $r$-uniform hypergraph $\mathcal{H}$ only has $tr$ vertices in total, the matching $M$ must cover every vertex of $\mathcal{H}$. Furthermore, each edge of $M$ belongs to at least one of the matchings $M_{x_1,\dots,x_r}$ defined above, and must therefore be corresponding to an $r$-tuple in $(\{1,\dots,t-r\}\cup\{a_1,\dots,a_r\})^r$ that satisfies one of the rules (i)--(iv) above. In total, for the $t$ different $r$-tuples corresponding to the edges of the matching $M$, each of the symbols $1,\dots,t-r,a_1,\dots,a_r$ occurs exactly once in each of the $r$ positions.

For each $j \in \{1,\dots,r\}$, consider how we can obtain symbol $a_j$ outside the $j$-th position. The only way to achieve this is by taking an $r$-tuple from (ii). Thus, for every $j=1,\dots,r$, exactly one of the $r$-tuples corresponding to the edges of $M$ must have the form $(a_j,\dots,a_j,x_j,a_j,\dots,a_j)$, with $x_j$ in the $j$-th position, for some $x_j\in X_j$. Applying this argument for all $j=1,\dots,r$ specifies elements $x_1\in X_1,\dots,x_r\in X_r$ such that among the $r$-tuples corresponding to the edges of $M$ we have $(a_j,\dots,a_j,x_j,a_j,\dots,a_j)$ with $x_j$ in the $j$-th position for every $j=1,\dots,r$. 

Note that besides these $r$ specific $r$-tuples, none of the other $r$-tuples corresponding to the edges of $M$ can be generated by rule (ii). Indeed, whenever we apply rule (ii) with some $j=1,\dots,r$, we obtain an $r$-tuple with $a_j$ in all positions outside the $j$-th position. Since $a_j$ can only appear once in these positions among the $r$-tuples corresponding to the edges of $M$, besides the specific $r$-tuples $(a_j,\dots,a_j,x_j,a_j,\dots,a_j)$ with $x_j$ in the $j$-th position for $j=1,\dots,r$ (for the specific elements $x_1\in X_1,\dots, x_r\in X_r$ above) there cannot be any other $r$-tuples obtained from rule (ii).

On the other hand, for every $j=1,\dots,r$, each element $x_j'\in X_j\setminus \{x_j\}$  also needs to appear in the $j$-th position of some $r$-tuple corresponding to an edge of $M$. We observed above that the only way to achieve this is to use an $r$-tuple obtained from rule (ii) or (iv), but we just concluded that we cannot use rule (ii) anymore. This means that for every $j=1,\dots,r$ and every element $x_j'\in X_j\setminus \{x_j\}$ we have to use the $r$-tuple $(x_j',\dots,x_j')$ in  rule (iv) in order to get $x_j'$ into the $j$-th position. In other words, for every $j=1,\dots,r$ and every $x_j'\in X_j\setminus \{x_j\}$ the $r$-tuple $(x_j',\dots,x_j')$ needs to be one of the $r$-tuples corresponding to the edges of $M$.

Now, we claim that any $r$-tuple corresponding to an edge of $M$ that was obtained by rule (iii) can only be part of the matching $M_{x_1,\dots,x_r}$ (and no other matching $M_{x_1',\dots,x_r'}$). Indeed, suppose among the $r$-tuples corresponding to the edges of $M$ we have the $r$-tuple $\sigma^h(x_1',\dots,x_r')$ from rule (iii) for some $h\in \{1,\dots,r-1\}$ and some matching $M_{x_1',\dots,x_r'}$ with $(x_1',\dots,x_r')\in X_1\times\dots\times X_r$ such that $(x_1,\dots,x_r)\neq (x_1',\dots,x_r')$. Then there must be some index $j\in \{1,\dots,r\}$ such that $x_j'\neq x_j$ and hence $x_j'\in X_j\setminus \{x_j\}$. But then the $r$-tuple $(x_j',\dots,x_j')$ also corresponds to an edge of $M$, and it already has $x_j'$ in all positions. So we cannot also have $\sigma^h(x_1',\dots,x_r')$ (which also contains $x_j'$ in some position) among the $r$-tuples corresponding to the edges of $M$. This contradiction shows that indeed all $r$-tuples corresponding edges of $M$ obtained by rule (iii) are only part of the matching $M_{x_1,\dots,x_r}$. In particular, since $M$ is a rainbow matching, this means that there can be at most one $r$-tuple obtained by rule (iii) among the $r$-tuples corresponding to edges of $M$.

Recall that among the $r$-tuples corresponding to the edges of $M$ we have $(a_j,\dots,a_j,x_j,a_j,\dots,a_j)$ (with $x_j$ in the $j$-th position) for every $j=1,\dots,r$. This means that for each $j \in \{1,\dots,r\}$, we cannot have the $r$-tuple $(x_j,\dots,x_j)$ in (iv) among the $r$-tuples corresponding to the edges of $M$ anymore (otherwise, $x_j$ would appear in the $j$-th position more than once). However, for the edges of $M$ we still need to have $r$-tuples containing each $x_j$ for $j=1,\dots,r$ in all other positions besides the $j$-th position. In particular, for $j=2,\dots,r$, we need to have an $r$-tuple containing $x_j$ in the first position. The only way to achieve this is to use an $r$-tuple obtained by rule (iii) (indeed, we just saw that we cannot use rule (iv) for this, and note that rule (ii) only gives entries from $X_j$ in the $j$-th position). Thus, among the $r$-tuples corresponding to the edges of $M$, there need to be at least $r-1\ge 2$ different $r$-tuples obtained by rule (iii). This is a contradiction, since we previously proved that there can be at most one such $r$-tuple.

This contradiction shows that there cannot be a rainbow matching $M$ of size $t$ for the matchings $M_{x_1,\dots,x_r}$ for $(x_1,\dots,x_r)\in X_1\times\dots\times X_r$. This finishes the proof of Theorem \ref{thm-construction-fixed-r}.
\end{proof}

\section{Fixed matching sizes}
\label{sect-fixed-t}

In this section, we prove Theorems \ref{thm-fixed-t-F} and \ref{thm-fixed-t-f}. More specifically, we prove Theorem \ref{thm-fixed-t-F} in the first subsection, the lower bound in Theorem \ref{thm-fixed-t-f} in the second subsection, and the upper bound  in Theorem \ref{thm-fixed-t-f} in the third subsection. The lower bounds in both theorems are based on probabilistic arguments, but in the last subsection we give some simple explicit constructions (which give somewhat weaker bounds for fixed $t$ and large $r$, but better bounds if $t$ and $r$ are both large).

\subsection{Proof of Theorem \ref{thm-fixed-t-F}}
\label{subsect-fixed-t-F}

We can deduce Theorem \ref{thm-fixed-t-F} from the following proposition, which is a special case of a proposition due to Lov\'asz and the second author \cite{lovasz-sauermann} (to be precise, it corresponds to the case $m=2$ in \cite[Proposition 3.1]{lovasz-sauermann}). The proof of this proposition builds upon the approach in earlier work of Kleinberg--Sawin--Speyer \cite{kleinberg-sawin-speyer} in the case of $t=3$. However, in the special case stated below the proof becomes much simpler and shorter (similar to the proof of Theorem \ref{thm-fixed-t-f-stronger} below, while the proof of the more general proposition in \cite{lovasz-sauermann} takes up around 40 pages).

\begin{proposition}\label{prop-t-color-sum-free}
For every fixed $t\ge 3$, there exists a constant $C_t\ge 0$ such that the following holds. For every positive number $n$ which is divisible by $t$, there exists a collection of $t$-tuples $(x_{1,j},x_{2,j},\dots,x_{t,j})\in \{0,1\}^n\times \dots\times \{0,1\}^n$ for $j=1,\dots,N$ with
\[N\ge \left(\frac{t^t}{(t-1)^{t-1}}\right)^{n/t-C_t\sqrt{n}}\]
such that the following conditions hold:
\begin{itemize}
\item For all $j_1,\dots,j_t\in \{1,\dots,N\}$, we have
\[x_{1,j_1}+\dots+x_{t,j_t}=\mathbbm{1}^n \text{ in }\mathbb{R}^n\quad\text{if and only if}\quad j_1=j_2=\dots=j_t.\]
\item For all $i=1,\dots,t$ and $j=1,\dots,N$, the vector $x_{i,j}\in \{0,1\}^n$ consists of precisely $n/t$ ones.
\end{itemize}
\end{proposition}

Here, $\mathbbm{1}^n$ denotes the all-ones vector. Let us remark that the lower bound for $N$ in \cite[Proposition 3.1]{lovasz-sauermann} is written as $N\ge \Gamma_{2,t}^{n-O(\sqrt{n})}$ (where the implicit constant in the $O$-notation may depend on $t$), but by the definition of $\Gamma_{2,t}$ from \cite[pp. 1]{lovasz-sauermann} we have
\[\Gamma_{2,t}=\min_{0<\gamma<1} \frac{1+\gamma}{\gamma^{1/t}}= \frac{1+1/(t-1)}{(t-1)^{-1/t}}=\frac{t}{(t-1)^{(t-1)/t}}\]
and hence we indeed obtain
\[N\ge \Gamma_{2,t}^{n-O(\sqrt{n})}=\left( \frac{t}{(t-1)^{(t-1)/t}}\right)^{n-O(\sqrt{n})}=\left( \frac{t^t}{(t-1)^{t-1}}\right)^{n/t-O(\sqrt{n})}\]
as stated above.

Let us now deduce Theorem \ref{thm-fixed-t-F} from Proposition \ref{prop-t-color-sum-free}. In fact, this deduction gives the following stronger statement, which immediately implies Theorem \ref{thm-fixed-t-F}.

\begin{theorem}\label{thm-fixed-t-F-stronger}
For every fixed $t\ge 2$, there exists a constant $D_t\ge 0$ such that the  following holds. For any $r\ge 2$, there exist matchings $M_1,\dots,M_N$ of size $t$ in some $r$-uniform hypergraph with $tr$ vertices, where
\[N\ge \left(\frac{t^t}{(t-1)^{t-1}}\right)^{r-D_t\sqrt{r}}\]
and such that for any pairwise disjoint edges $e_1\in M_{j_1},\dots,e_t\in M_{j_t}$ (for some indices $j_1,\dots,j_t\in \{1,\dots,N\}$) we must have $j_1=\dots=j_t$.
\end{theorem}

Note that in order for $M_1,\dots,M_N$ not to have a rainbow matching (which suffices for deducing Theorem \ref{thm-fixed-t-F}), it would already be sufficient to have ``at least two of the indices $j_1,\dots,j_t$ must agree'' instead of ``we must have $j_1=\dots=j_t$'' at the end of the statement. In this sense, Theorem \ref{thm-fixed-t-F-stronger} is significantly stronger than  Theorem \ref{thm-fixed-t-F}.

\begin{proof}[Proof of Theorem \ref{thm-fixed-t-F-stronger}]
First, let us consider the case that $t=2$. For any $r\ge 2$, let us consider a complete $r$-uniform hypergraph on $2r$ vertices. The edges of this hypergraph can be partitioned into $N=\binom{2r}{r}/2$ matchings of size $2$ (one matching for each of the ways to split the vertex set into two subsets of size $r$). These matchings satisfy the conditions in Theorem \ref{thm-fixed-t-F-stronger}, since any two disjoint edges indeed belong to the same matching and since $N=\binom{2r}{r}/2\ge 2^{2r}/(4r+2)\ge 4^r/(4\sqrt{r})^2\ge 4^r/(2^{4\sqrt{r}})^2= (2^2/1^1)^{r-4\sqrt{r}}$. Thus, the statement in Theorem \ref{thm-fixed-t-F-stronger} for $t=2$ holds with $D_2=4$.

Let us now assume that $t\ge 3$, and define $D_t=C_t\cdot \sqrt{t}+1$ for the constant $C_t\ge 0$ in Proposition \ref{prop-t-color-sum-free}. Let us apply Proposition \ref{prop-t-color-sum-free} with $n=t(r-1)$. We obtain a a collection of $t$-tuples $(x_{1,j},x_{2,j},\dots,x_{t,j})\in \{0,1\}^n\times \dots\times \{0,1\}^n$ for $j=1,\dots,N$ with
\[N\ge \left(\frac{t^t}{(t-1)^{t-1}}\right)^{n/t-C_t\sqrt{n}}\ge \left(\frac{t^t}{(t-1)^{t-1}}\right)^{r-1-C_t\sqrt{tr}}\ge \left(\frac{t^t}{(t-1)^{t-1}}\right)^{r-D_t\sqrt{r}}\]
satisfying the conditions in Proposition \ref{prop-t-color-sum-free}.

To construct the desired matchings $M_1,\dots,M_N$, let us consider the complete $r$-uniform hypergraph on $tr$ vertices labeled by $1,\dots,n$ and $a_1,\dots,a_t$ (recall that $n=t(r-1)$). For $j=1,\dots,N$, let us define a matching $M_j$ of size $t$ as follows. By the second condition in Proposition \ref{prop-t-color-sum-free}, for $i=1,\dots,t$ the vector $x_{i,j}\in \{0,1\}^n$ consists of precisely $n/t=r-1$ ones, so it corresponds to a subset $X_{i,j}\su \{1,\dots,n\}$ of size $r-1$. Now, let the matching $M_j$ consist of the $t$ edges $\{a_i\}\cup X_{i,j}$ for $i=1,\dots,t$.

Each of these edges has size $r$, so in order to check that $M_j$ is indeed a matching, it suffices to check that $\bigcup_{i=1}^{t} (\{a_i\}\cup X_{i,j})=\{1,\dots,n\}\cup \{a_1,\dots,a_t\}$. Indeed, by the first condition in Proposition \ref{prop-t-color-sum-free} we have $x_{1,j}+\dots+x_{t,j}=\mathbbm{1}^n$ and hence $\bigcup _{i=1}^{t} X_{i,j}=\{1,\dots,n\}$. So we indeed obtain $\bigcup_{i=1}^{t} (\{a_i\}\cup X_{i,j})=\{a_1,\dots,a_t\}\cup \bigcup _{i=1}^{t} X_{i,j}=\{1,\dots,n\}\cup \{a_1,\dots,a_t\}$ and $M_j$ is a matching for each $j \in \{1,\dots,n\}$.

It remains to show that for any pairwise disjoint edges $e_1\in M_{j_1},\dots,e_t\in M_{j_t}$ (for some indices $j_1,\dots,j_t\in \{1,\dots,N\}$) we must have $j_1=\dots=j_t$. Each edge contains precisely one element from $\{a_1,\dots,a_t\}$, and the edges $e_1,\dots,e_t$ must contain distinct elements from $\{a_1,\dots,a_t\}$ since they are pairwise disjoint. Upon relabeling, we may assume that $e_i$ contains $a_i$ for $i=1,\dots,t$. Then for $i=1,\dots,t$, we have $e_i=\{a_i\}\cup X_{i,j_i}$, where $X_{i,j_i}\su \{1,\dots,n\}$ is the subset (of size $r-1$) corresponding to ones the vector $x_{i,j_i}\in \{0,1\}^n$. Since $e_1,\dots,e_t$ are $t$ pairwise disjoint subsets of size $r$ of the vertex set of  $\{1,\dots,n\}\cup \{a_1,\dots,a_t\}$ of size $tr$, we have $\bigcup_{i=1}^{t} (\{a_i\}\cup X_{i,j_i})=e_1\cup\dots\cup e_t=\{1,\dots,n\}\cup \{a_1,\dots,a_t\}$. Thus, $\bigcup _{i=1}^{t} X_{i,j_i}=\{1,\dots,n\}$ and therefore (as the sets $X_{i,j_i}\su e_i$ for $i=1,\dots,t$ are pairwise disjoint) we have $x_{1,j_1}+\dots+x_{t,j_t}=\mathbbm{1}^n$. But by the first condition in Proposition \ref{prop-t-color-sum-free} this implies that $j_1=\dots=j_t$, as desired.
\end{proof}

\subsection{Proof of the lower bound in Theorem \ref{thm-fixed-t-f}}
\label{subsect-fixed-t-f}

In this subsection, we prove the lower bound in Theorem  \ref{thm-fixed-t-f} for $r$-partite hypergraphs. Again, we prove a significantly stronger statement (which immediately implies the lower bound in Theorem  \ref{thm-fixed-t-f}):

\begin{theorem}\label{thm-fixed-t-f-stronger}
For every fixed $t\ge 2$, there exists a constant $D_t\ge 0$ such that the  following holds. For any $r\ge 2$, there exist matchings $M_1,\dots,M_N$ of size $t$ in some $r$-partite $r$-uniform hypergraph with $tr$ vertices, where
\[N\ge t^{r-D_t\sqrt{r}}\]
and such that for any pairwise disjoint edges $e_1\in M_{j_1},\dots,e_t\in M_{j_t}$ (for some indices $j_1,\dots,j_t\in \{1,\dots,N\}$) we must have $j_1=\dots=j_t$.
\end{theorem}

We prove this theorem with a very similar strategy as in the proof of Lov\'asz and the second author \cite{lovasz-sauermann} for what is stated as Proposition \ref{prop-t-color-sum-free} above, which is in turn very similar to the proof of Kleinberg--Sawin--Speyer \cite{kleinberg-sawin-speyer} in the case of $t=3$.

In the proof of Theorem \ref{thm-fixed-t-f-stronger}, we will use the following lemma. This lemma appears as \cite[Lemma 3.4]{lovasz-sauermann} and its proof relies on a modification of Behrend's construction \cite{behrend} due to Alon \cite[Lemma 3.1]{alon}.

\begin{lemma}\label{lemma-behrend}
For every integer $t\ge 3$ and every prime $P\ge t$, there is a collection of $t$-tuples $(y_{1,h},y_{2,h},\dots,y_{t,h})\in \mathbb{F}_P\times \dots\times \mathbb{F}_P$ for $h=1,\dots,R$ with $R\ge P\cdot \exp(-12\sqrt{\ln P\ln t})$
such that for all $h_1,\dots,h_t\in \{1,\dots,R\}$ we have
\begin{equation}\label{eq-behrend}
y_{1,h_1}+\dots+y_{t,h_t}=0\text{ in }\mathbb{F}_P\quad\text{if and only if}\quad h_1=h_2=\dots=h_t.
\end{equation}
\end{lemma}

Note that the conditions in this lemma in particular imply that for each $i \in \{1,\dots,t\}$ the vectors $y_{i,1},\dots,y_{i,R}$ are distinct (indeed, if $y_{i,h}=y_{i,h'}$ for $h\neq h'$, then we have $y_{1,h}+\dots+y_{i-1,h}+y_{i,h'}+y_{i+1,h}+\dots+y_{t,h}=y_{1,h}+\dots+y_{t,h}=0$ contradicting the condition).

In order to prove Theorem \ref{thm-fixed-t-f-stronger}, we will use a randomized procedure to sample certain matchings in a complete $r$-partite $r$-uniform hypergraph with $t$ vertices in each part. This procedure will use a collection of $t$-tuples $(y_{1,h},y_{2,h},\dots,y_{t,h})$ as in Lemma \ref{lemma-behrend}, for a suitably chosen prime $P$. To describe the procedure, we make the following definitions.

\begin{setting}\label{setting}
    Let $t\ge 3$ and $r\ge 2$ be integers. Fix a prime $P$ such that
    \begin{equation}\label{eq-def-P}
2t^{t+1}\cdot (t-1)!^{(r-1)/(t-2)}\le P\le 4t^{t+1}\cdot (t-1)!^{(r-1)/(t-2)},
\end{equation}
and a collection of $t$-tuples $(y_{1,h},y_{2,h},\dots,y_{t,h})\in \mathbb{F}_P\times \dots\times \mathbb{F}_P$ for $h=1,\dots,R$ as in Lemma \ref{lemma-behrend}. For $i=1,\dots,t$, define $Y_i=\{y_{i,h}\mid 1\le h\le R\}$.

Furthermore, let $Z\su\{0,1\}^{tr}\su \mathbb{F}_P^{tr}$ be the set of vectors $z\in \{0,1\}^{tr}$ such that the restriction of $z$ to each of the index subsets $\{1,\dots,t\}, \{t+1,\dots, 2t\}, \dots, \{t(r-1)+1,\dots, tr\}$ has exactly one $1$-entry. Consider the partition $Z=Z_1\cup\dots\cup Z_t$ given by defining $Z_i$ for $i=1,\dots,t$ to be the set of vectors $z\in Z$ such that the $i$-th entry of $z$ is $1$.

Now, let $f:\mathbb{F}_P^{tr}\to \mathbb{F}_P$ a uniformly random $\mathbb{F}_P$-linear map with the condition that $f(\mathbbm{1}^{tr})=0$.

Let us say that a $t$-tuple $(z_1,\dots,z_t)\in Z_1\times \dots\times Z_t$ is a \emph{candidate $t$-tuple} if $z_1+\dots+z_t=\mathbbm{1}^{tr}$ and $f(z_i)\in Y_i$ for $i=1,\dots,t$. Let us sat that a candidate $t$-tuple $(z_1,\dots,z_t)\in Z_1\times \dots\times Z_t$ is \emph{isolated} if there is no other candidate $t$-tuple $(z_1',\dots,z_t')\in Z_1\times \dots\times Z_t$ such that $z_i'=z_i$ for some $i\in \{1,\dots,t\}$.
\end{setting}

We remark that for any $t\ge 3$, finding a prime $P$ satisfying (\ref{eq-def-P}) is possible by Bertrand's postulate. Note that the notions of a candidate $t$-tuple and an isolated candidate $t$-tuple depend on the outcome of the random function $f$ (since the sets $Z_1',\dots,Z_t'$ depend on the outcome of $f$).

To motivate this setup, we start with a few simple observations. First, note that the vectors in $Z$ are in correspondence to the edges of a complete $r$-partite $r$-uniform hypergraph with $t$ vertices in each part. Indeed, let us consider a complete $r$-partite $r$-uniform hypergraph with vertex set $\{1,\dots,tr\}$, partitioned into the $r$ parts $\{1,\dots,t\},\allowbreak  \{t+1,\dots, 2t\}, \dots, \{t(r-1)+1,\dots, tr\}$, each of size $t$. The indicator vectors of the edges of this hypergraph are precisely the vectors in $Z$. To see this, note that the elements of $Z\su\{0,1\}^{tr}\su \mathbb{F}_P^{tr}$ are precisely those vectors with exactly $r$ ones and $tr-r$ zeroes, which have exactly one $1$-entry among the first $t$ indices (i.e.\ among the index set $\{1,\dots,t\}$), exactly one $1$-entry among the next $t$ indices (i.e.\ among the index set $\{t+1,\dots,2t\}$), and so on.

Now, a matching of size $t$ in this hypergraph corresponds to a collection $\{z_1,\dots,z_t\}\su Z$ of $t$ vectors in $Z$ with $z_1+\dots+z_t=\mathbbm{1}^{tr}$. Note that here it does not matter whether the equation $z_+\dots+z_t=\mathbbm{1}^{tr}$ is considered over $\mathbb{F}_P$ or over the reals. Indeed, each of the vectors $z_1,\dots,z_t\in\{0,1\}^{tr}$ consists of precisely $r$ ones, and so the $t$ vectors sum to $\mathbbm{1}^{tr}$ over $\mathbb{F}_P$ (or over the reals) if and only if their supports are a partition of $\{1,\dots,tr\}$.

For any $\{z_1,\dots,z_t\}\su Z$ with $z_1+\dots+z_t=\mathbbm{1}^{tr}$ (i.e. for any matching of size $t$ in the hypergraph), each of the (pairwise disjoint) sets $Z_1,\dots,Z_t$ must contain precisely one of the vectors $z_1,\dots,z_t$. Thus, upon relabeling the indices, each matching of size $t$ in the hypergraph corresponds to a $t$-tuple of vectors $(z_1,\dots,z_t)\in Z_1\times\dots\times Z_t$ with $z_1+\dots+z_t=\mathbbm{1}^{tr}$. To obtain the desired matchings in Theorem \ref{thm-fixed-t-f-stronger}, we will consider the matchings corresponding to isolated candidate $t$-tuples $(z_1,\dots,z_t)\in Z_1\times\dots\times Z_t$. It is therefore important to show that (at least for some outcome of the random map $f$) there are many isolated candidate $t$-tuples.

The following lemma gives a lower bound for the probability that a given $t$-tuple $(z_1,\dots,z_t)\in Z_1\times\dots\times Z_t$ with $z_1+\dots+z_t=\mathbbm{1}^{tr}$ is an isolated candidate $t$-tuple. Proving this lemma is the hardest part of the proof of Theorem \ref{thm-fixed-t-f-stronger}.

\begin{lemma}\label{lemma-probability-isolated-candidate}
Consider Setting \ref{setting}. Suppose that $(z_1,\dots,z_t)\in Z_1\times\dots\times Z_t$ satisfies $z_1+\dots+z_t=\mathbbm{1}^{tr}$. Then the probability that $(z_1,\dots,z_t)$ is an isolated candidate $t$-tuple is at least $R/(2P^{t-1})$.
\end{lemma}

Before proving this lemma, let us first show how it can be used to prove Theorem \ref{thm-fixed-t-f-stronger}.

\begin{proof}[Proof of Theorem \ref{thm-fixed-t-f-stronger} assuming Lemma \ref{lemma-probability-isolated-candidate}]
Let us consider the vertex set $\{1,\dots,tr\}$, partitioned into the $r$ parts $\{1,\dots,t\},\allowbreak  \{t+1,\dots, 2t\}, \dots, \{t(r-1)+1,\dots, tr\}$, each of size $t$. Let $\mathcal{H}$ be the the complete $r$-partite $r$-uniform hypergraph on the vertex set $\{1,\dots,tr\}$ with this $r$-partition.

Let us first treat the case $t=2$. Then $\mathcal{H}$ has $2^r$ edges. These edges can be divided into $N=2^r/2$ matchings of size $2$ (each consisting of an edge and its complement). These matchings satisfy the conditions in Theorem \ref{thm-fixed-t-f-stronger}, since any two disjoint edges indeed belong to the same matching and since $N=2^r/2 \ge 2^{r-\sqrt{r}}$. Thus, the statement in Theorem \ref{thm-fixed-t-f-stronger} for $t=2$ holds with $D_2=1$.

So let us from now on assume that $t\ge 3$, and consider Setting \ref{setting}. Recall from the above discussion that the edges of the hypergraph $\mathcal{H}$ correspond precisely to the vectors in $Z$ and that the matchings of size $t$ in $\mathcal{H}$ correspond precisely to the $t$-tuples $(z_1,\dots,z_t)\in Z_1\times\dots\times Z_t$ with $z_1+\dots+z_t=\mathbbm{1}^{tr}$.

Note that $|Z|=t^r$ and $|Z_1|=|Z_2|=\dots=|Z_t|=t^{r-1}$. Furthermore, the number of $t$-tuples $(z_1,\dots,z_t)\in Z_1\times\dots\times Z_t$ with $z_1+\dots+z_t=\mathbbm{1}^{tr}$ is precisely $t!^{r-1}$ (indeed, this is the total number of matchings $(e_1,\dots,e_t)$ in the hypergraph $\mathcal{H}$ such that $i\in e_i$ for $i=1,\dots,t$, since for each part the $r$-partition of the vertex set of $\mathcal H$ except the first part there are exactly $t!$ possibilities how the $t$ elements of the part are distributed between $e_1,\dots,e_t$).

By Lemma \ref{lemma-probability-isolated-candidate}, each of these $t!^{r-1}$ different $t$-tuples $(z_1,\dots,z_t)\in Z_1\times\dots\times Z_t$ with $z_1+\dots+z_t=\mathbbm{1}^{tr}$ is an isolated candidate $t$-tuple with probability at least $R/(2P^{t-1})$. Thus, the expected number of isolated candidate $t$-tuples $(z_1,\dots,z_t)\in Z_1\times\dots\times Z_t$ is at least
\[t!^{r-1}\cdot \frac{R}{2P^{t-1}}.\]

Note that we have $(t-1)!=2\dotsm (t-1)\le t^{t-2}$ and hence $P\le 4t^{t+1}\cdot (t-1)!^{(r-1)/(t-2)}\le t^{2t+1}\cdot t^{r-1}=t^{2t+r}$, so
\[R\ge P\cdot \exp\Big(-12\sqrt{\ln P\ln t}\Big)\ge P\cdot \exp\Big(-12\sqrt{2t+r} \cdot \ln t\Big)\ge \frac{P}{t^{24t+12\sqrt{r}}}.\]

Thus, choosing $D_t=t^2+26t+13$, the expected number of isolated candidate $t$-tuples in $Z_1\times\dots\times Z_t$ is at least
\[t!^{r-1}\cdot \frac{R}{2P^{t-1}}\ge t!^{r-1}\cdot \frac{1}{2P^{t-2}}\cdot \frac{1}{t^{24t+12\sqrt{r}}}\ge \frac{t!^{r-1}}{(t-1)!^{r-1}}\cdot \frac{1}{(4t^{t+1})^{t-2}\cdot 2t^{24t+12\sqrt{r}}}\ge t^{r-1}\cdot \frac{1}{t^{t^2+26t+12\sqrt{r}}}\ge t^{r-D_t\sqrt{r}},\]
where for the second inequality we used the upper bound for $P$ in (\ref{eq-def-P}). So let us fix an outcome of the random map $f$ (see Setting \ref{setting}) such that there are at least $t^{r-D_t\sqrt{r}}$ isolated candidate $t$-tuples.

Each isolated candidate $t$-tuple $(z_1,\dots,z_t)\in Z_1\times \dots\times Z_t$ satisfies $z_1+\dots+z_t=\mathbbm{1}^{tr}$, and so it corresponds to a matching of size $t$ in $\mathcal{H}$. Let $M_1,\dots,M_N$ be the resulting matchings of size $t$ in the hypergraph $\mathcal{H}$, where $N\ge t^{r-D_t\sqrt{r}}$ is the number of isolated candidate $t$-tuples. It remains to show that the condition at the end of the statement of Theorem  \ref{thm-fixed-t-f-stronger} holds. This is the content of the following claim.

\begin{claim}\label{claim-isolated-candidate-matchings}
For any pairwise disjoint edges $e_1\in M_{j_1},\dots,e_t\in M_{j_t}$ (for some indices $j_1,\dots,j_t\in \{1,\dots,N\}$), we must have $j_1=\dots=j_t$
\end{claim}
\begin{proof}
Since the edges $e_1,\dots,e_t$ are pairwise disjoint and each of them contains exactly one vertex from the first part $\{1,\dots,t\}$ of the $r$-partition of $\mathcal{H}$, upon relabeling we may assume that for $i=1,\dots,t$ we have $i\in e_i$. Let $z_1',\dots,z_t'\in Z\su \{0,1\}^{tr}$ be the vectors corresponding to $e_1,\dots,e_t$, then we have $z_i'\in Z_i$ for $i=1,\dots,t$. In other words, we have $(z_1',\dots,z_t')\in Z_1\times \dots,\times Z_t$. Since $e_1,\dots,e_t$ form a matching, we furthermore have $z_1'+\dots+z_t'=\mathbbm{1}^{tr}$.

We claim that $f(z_i')\in Y_i$ for $i=1,\dots,t$. Indeed, consider some $i\in \{1,\dots, t\}$ and recall that $z_i'$ corresponds to the edge $e_i\in M_{j_i}$. The matching $M_{j_i}$ corresponds to some isolated candidate $t$-tuple $(z_1,\dots,z_t)\in Z_1\times \dots\times Z_t$, and, since $e_i$ is an edge of $M_{j_i}$, one of the vectors $z_1,\dots,z_t$ must be equal to $z_i'$. Because $z_i'\in Z_i$ (and $Z_1,\dots,Z_t$ are pairwise disjoint), we must have $z_i=z_i'$. So we can conclude that $f(z_i')=f(z_i)\in Y_i$ (as $(z_1,z_2,\dots,z_t)$ is a candidate $t$-tuple).

Now we established that $(z_1',\dots,z_t')\in Z_1\times \dots,\times Z_t$ satisfies $z_1'+\dots+z_t'=\mathbbm{1}^{tr}$ and $f(z_i')\in Y_i$ for $i=1,\dots,t$. This means that $(z_1',\dots,z_t')$ is a candidate $t$-tuple.

We claim that for each $i \in \{1,\dots,t\}$ the isolated candidate $t$-tuple corresponding to the matching $M_{j_i}$ must be precisely $(z_1',\dots,z_t')$. Indeed, consider some $i\in \{1,\dots, t\}$  and let $(z_1,\dots,z_t)\in Z_1\times \dots\times Z_t$ be the isolated candidate $t$-tuple corresponding to $M_{j_i}$. Recalling that $e_i$ is one of the edges of $M_{j_i}$ and that $e_i$ corresponds to the vector $z_i'$, we can see that $z_i'=z_i$. If we had $(z_1,\dots,z_t)\neq (z_1',\dots,z_t')$, then the candidate $t$-tuple $(z_1,\dots,z_t)$ would not be isolated (since $(z_1',\dots,z_t')$ is a candidate $t$-tuple satisfying $z_i'=z_i$). Hence we must have $(z_1,\dots,z_t)= (z_1',\dots,z_t')$, and the isolated candidate $t$-tuple corresponding to the matching $M_{j_i}$ is indeed $(z_1',\dots,z_t')$.

We proved that all of the matchings $M_{j_1},\dots,M_{j_t}$ correspond to the same isolated candidate $t$-tuple, namely to $(z_1',\dots,z_t')$. This means that the matchings $M_{j_1},\dots,M_{j_t}$ are all equal and we have $j_1=\dots=j_t$.
\end{proof}

This finishes the proof of Theorem \ref{thm-fixed-t-f-stronger}.
\end{proof}

It remains to prove Lemma \ref{lemma-probability-isolated-candidate}, which is the aim of the remainder of this subsection. We adopt the definitions and conventions in Setting \ref{setting} throughout the rest of this subsection. 

We prepare the proof of Lemma \ref{lemma-probability-isolated-candidate} with a series of claims. 

\begin{claim}\label{claim-linear-map-randomness}
Suppose $x_1,\dots,x_{\ell}\in \mathbb{F}_P^{tr}$ are vectors such that $x_1,\dots,x_{\ell}$ and $\mathbbm{1}^{tr}$ are linearly independent over $\mathbb{F}_P$. Then for the random linear map $f:\mathbb{F}_P^{tr}\to\mathbb{F}_P$ as above (with the condition that $f(\mathbbm{1}^{tr})=0$), the images $f(x_1),\dots,f(x_{\ell})$ are independent uniformly random elements of $\mathbb{F}_P$.
\end{claim}
\begin{proof}
We can extend $x_1,\dots,x_{\ell},\mathbbm{1}^{tr}$ to some basis of $\mathbb{F}_P^{tr}$, and model the choice of the linear map $f:\mathbb{F}_P^{tr}\to\mathbb{F}_P$ by choosing independent uniformly random images in $\mathbb{F}_P$ for all basis vectors except $\mathbbm{1}^{tr}$ (for $\mathbbm{1}^{tr}$, we have to choose $f(\mathbbm{1}^{tr})=0$). With this description for the map $f$, the claim is immediate.
\end{proof}

Our next claim gives a helpful characterization of candidate $t$-tuples in terms of the $t$-tuples $(y_{1,h},y_{2,h},\dots,y_{t,h})\in \mathbb{F}_P\times \dots\times \mathbb{F}_P$ from Lemma \ref{lemma-behrend}.

\begin{claim}\label{claim-candidate-characterization}
A $t$-tuple $(z_1,\dots,z_t)\in Z_1\times\dots\times Z_t$ satisfying $z_1+\dots+z_t=\mathbbm{1}^{tr}$ is a candidate $t$-tuple if and only if there is some $h\in \{1,\dots,R\}$ such that $(f(z_1),\dots,f(z_t))=(y_{1,h},y_{2,h},\dots,y_{t,h})$.
\end{claim}
\begin{proof}
Since we already assumed that $(z_1,\dots,z_t)\in Z_1\times\dots\times Z_t$ and $z_1+\dots+z_t=\mathbbm{1}^{tr}$, the $t$-tuple $(z_1,\dots,z_t)$ is a candidate $t$-tuple if and only if we have $f(z_i)\in Y_i$ for $i=1,\dots,t$.

If $(f(z_1),\dots,f(z_t))=(y_{1,h},y_{2,h},\dots,y_{t,h})$ for some $h\in \{1,\dots,R\}$, then we indeed have $f(z_i)=y_{i,h}\in Y_i$ for all $i=1,\dots,t$  (recalling the definition of $Y_i$ in Setting \ref{setting}).

For the reverse direction, assume that $f(z_i)\in Y_i$ for $i=1,\dots,t$. This means that for each $i \in \{1,\dots,t\}$, there is an index $h_i\in \{1,\dots,R\}$ such that $f(z_i)=y_{i,h_i}$. We now have $y_{1,h_1}+\dots+y_{t,h_t}=f(z_1)+\dots+f(z_t)=f(z_1+\dots+z_t)=f(\mathbbm{1}^{tr})=0$ in $\mathbb{F}_P$. By the condition in Lemma \ref{lemma-behrend} this means that $h_1=\dots=h_t$. So there exists some $h\in \{1,\dots,R\}$ such that $f(z_i)=y_{i,h}$ for $h=1,\dots,t$, meaning that $(f(z_1),\dots,f(z_t))=(y_{1,h},y_{2,h},\dots,y_{t,h})$.
\end{proof}

By combining the two previous claims, we can determine the probability that a given $t$-tuple $(z_1,\dots,z_t)\in Z_1\times\dots\times Z_t$ with $z_1+\dots+z_t=\mathbbm{1}^{tr}$ is a candidate $t$-tuple.

\begin{claim}\label{claim-probability-single-candidate}
Suppose that $(z_1,\dots,z_t)\in Z_1\times\dots\times Z_t$ satisfies $z_1+\dots+z_t=\mathbbm{1}^{tr}$. Then the probability that $(z_1,\dots,z_t)$ is a candidate $t$-tuple is precisely $R/P^{t-1}$.
\end{claim}
\begin{proof}

First, note that the vectors $z_1,\dots,z_{t-1},z_t\in \mathbb{F}_P^{tr}$, are linearly independent. Indeed, when restricting the index set to $\{1,\dots,t\}$, these vectors are the standard basis vectors in $\mathbb{F}_P^{t}$. Since $z_1+\dots+z_t=\mathbbm{1}^{tr}$, this implies that the vectors $z_1,\dots,z_{t-1}$ and $\mathbbm{1}^{tr}$ are linearly independent over $\mathbb{F}_P$. Hence, by Claim \ref{claim-linear-map-randomness}, the images $f(z_1),\dots,f(z_{t-1})$ are independent uniformly random elements of $\mathbb{F}_P$.

By Claim \ref{claim-candidate-characterization} $(z_1,\dots,z_t)$ is a candidate $t$-tuple if and only if we have $(f(z_1),\dots,f(z_t))=(y_{1,h},y_{2,h},\dots,y_{t,h})$ for some $h\in \{1,\dots,R\}$.

We claim that for any $h\in \{1,\dots,R\}$, having $(f(z_1),\dots,f(z_t))=(y_{1,h},y_{2,h},\dots,y_{t,h})$ is equivalent to having $(f(z_1),\dots,f(z_{t-1}))=(y_{1,h},\dots,y_{t-1,h})$.
Indeed, if $(f(z_1),\dots,f(z_{t-1}))=(y_{1,h},\dots,y_{t-1,h})$ for some $h\in \{1,\dots,R\}$, then (using that $y_{1,h}+\dots+y_{t,h}=0$ by the condition in Lemma \ref{lemma-behrend}) we also have
\[f(z_t)=f(\mathbbm{1}^{tr}-z_1-\dots-z_{t-1})=f(\mathbbm{1}^{tr})-f(z_1)-\dots-f(z_{t-1})=0-y_{1,h}-\dots-y_{t-1,h}=y_{t,h},\]
and hence $(f(z_1),\dots,f(z_t))=(y_{1,h},y_{2,h},\dots,y_{t,h})$.

Recalling that the $t$-tuples $(y_{1,h},\dots,y_{t,h})$ for $h\in \{1,\dots,R\}$ are distinct (see the condition in Lemma \ref{lemma-behrend}), we can conclude that
\begin{align*}
\Pr[(z_1,\dots,z_t)\text{ is a candidate $t$-tuple}]&=\sum_{h=1}^{R}\Pr[(f(z_1),\dots,f(z_t))=(y_{1,h},\dots,y_{t,h})]\\
&=\sum_{h=1}^{R}\Pr[(f(z_1),\dots,f(z_{t-1}))=(y_{1,h},\dots,y_{t-1,h})]\\
&=\sum_{h=1}^{R}\Pr[f(z_1)=y_{1,h}]\dotsm \Pr[f(z_{t-1})=y_{t-1,h}]\\
&=\sum_{h=1}^{R} 1/P^{t-1}=R/P^{t-1},
\end{align*}
where in the third step we used that the images $f(z_1),\dots,f(z_{t-1})$ are probabilistically independent, and in the fourth step we used that these images are uniformly distributed in $\mathbb{F}_P$.
\end{proof}

In order to deduce Lemma \ref{lemma-probability-isolated-candidate} from Claim \ref{claim-probability-single-candidate} we need to give an upper bound for the probability that $(z_1,\dots,z_t)$ is a candidate $t$-tuple but not isolated. This happens precisely when there is another candidate $t$-tuple $(z_1',\dots,z_t')$ with $z_i'=z_i$ for some $i\in \{1,\dots,t\}$. The next claim bounds the probability that this happens for any given $(z_1,\dots,z_t)$ and $(z_1',\dots,z_t')$.

\begin{claim}\label{claim-prob-two-tuples-candidates}
Suppose that $(z_1,\dots,z_t)$ and $(z_1',\dots,z_t')$ are $t$-tuples in $Z_1\times\dots\times Z_t$ satisfying $z_1+\dots+z_t=\mathbbm{1}^{tr}$ and $z_1'+\dots+z_t'=\mathbbm{1}^{tr}$, such that $z_i=z_i'$ for some $i\in \{1,\dots,t\}$. Let $d=\dim \spn_{\mathbb{F}_P}(z_1,\dots,z_t,z_1',\dots,z_t')$. Then the probability that $(z_1,\dots,z_t)$ and $(z_1',\dots,z_t')$ are both candidate $t$-tuples is at most $R/P^{d-1}$.
\end{claim}

Here, $\spn_{\mathbb{F}_P}(z_1,\dots,z_t,z_1',\dots,z_t')$ denotes the span of $z_1,\dots,z_t,z_1',\dots,z_t'$, interpreted as vectors in $\mathbb{F}_P^{tr}$ (so, $d$ is the dimension of this span). We remark that in the above claim the probability that $(z_1,\dots,z_t)$ and $(z_1',\dots,z_t')$ are both candidate $t$-tuples is actually equal to $R/P^{d-1}$, but we will only need the upper bound for this probability as stated in the claim.

\begin{proof}[Proof of Claim \ref{claim-prob-two-tuples-candidates}]
If $(z_1,\dots,z_t)$ and $(z_1',\dots,z_t')$ are candidate $t$-tuples, then by Claim \ref{claim-candidate-characterization} there are $h,h'\in \{1,\dots,R\}$ such that $(f(z_1),\dots,f(z_t))=(y_{1,h},y_{2,h},\dots,y_{t,h})$ and $(f(z_1'),\dots,f(z_t'))=(y_{1,h'},y_{2,h'},\dots,y_{t,h'})$. However, recall that for some $i\in \{1,\dots,t\}$ we have $z_i=z_i'$ and therefore $y_{i,h}=f(z_i)=f(z_i')=y_{i,h'}$. Hence, since $y_{i,1},\dots,y_{i,R}$ are distinct (see the comment below Lemma \ref{lemma-behrend}), we must have $h=h'$.

Thus, if $(z_1,\dots,z_t)$ and $(z_1',\dots,z_t')$ are candidate $t$-tuples, then there must be some $h\in \{1,\dots,R\}$ such that $f(z_i)=f(z_i')=y_{i,h}$ for all $i=1,\dots,t$. So it suffices to show that for every $h\in \{1,\dots,R\}$ with probability at most $1/P^{d-1}$ we have $f(z_i)=f(z_i')=y_{i,h}$ for all $i=1,\dots,t$. So let us fix some $h\in \{1,\dots,R\}$.

Since the span of the vectors $z_1,\dots,z_t,z_1',\dots,z_t'$ in $\mathbb{F}_P^{tr}$ has dimension $d$, we can choose $d-1$ vectors among $z_1,\dots,z_t,z_1',\dots,z_t'$ such that these $d-1$ vectors together with $\mathbbm{1}^{tr}$ are linearly independent. Then by Claim \ref{claim-linear-map-randomness} the images of these $d-1$ vectors under the map $f$ are independent uniformly random elements of $\mathbb{F}_P$. Thus, the probability of having the desired images $f(z_i)=y_{i,h}$ or $f(z_i')=y_{i,h}$ for all of these $d-1$ elements is $1/P^{d-1}$. Hence the probability of having  $f(z_i)=f(z_i')=y_{i,h}$ for all $i=1,\dots,t$ is at most $1/P^{d-1}$, as desired.
\end{proof}

Our next claim gives some bounds for the dimension $d=\dim \spn_{\mathbb{F}_P}(z_1,\dots,z_t,z_1',\dots,z_t')$ occurring in Claim \ref{claim-prob-two-tuples-candidates}.

\begin{claim}\label{claim-values-of-d}
Consider $t$-tuples $(z_1,\dots,z_t),(z_1',\dots,z_t')\in Z_1\times\dots\times Z_t$ with $z_1+\dots+z_t=\mathbbm{1}^{tr}$ and $z_1'+\dots+z_t'=\mathbbm{1}^{tr}$, and let $d=\dim \spn_{\mathbb{F}_P}(z_1,\dots,z_t,z_1',\dots,z_t')$. If $(z_1,\dots,z_t)\neq (z_1',\dots,z_t')$, then we have $t+1\le d\le 2t-1$
\end{claim}
\begin{proof}
The upper bound $d=\dim \spn_{\mathbb{F}_P}(z_1,\dots,z_t,z_1',\dots,z_t')\le 2t-1$ follows from the fact that we have $z_1+\dots+z_t=\mathbbm{1}^{tr}=z_1'+\dots+z_t'$, so $z_1,\dots,z_t,z_1',\dots,z_t'$ cannot be linearly independent over $\mathbb{F}_P$.

For the lower bound $d\ge t+1$, recall when restricting the vectors $z_1,\dots,z_t$ to the first $t$ coordinates (i.e.\ when restricting the index set to $\{1,\dots,t\}$), these vectors become the standard basis vectors (this follows from the definition of the sets $Z_1,\dots,Z_t$ and the fact that $z_1+\dots+z_t=\mathbbm{1}^{tr}$). In particular, the vectors $z_1,\dots,z_t$ are linearly independent, and we have $\dim \spn_{\mathbb{F}_P}(z_1,\dots,z_t)= t$.

Furthermore, since $(z_1,\dots,z_t)\neq (z_1',\dots,z_t')$, there must be an index $i$ such that $z_i'\neq z_i$.  Suppose we could express $z_i'$ as a linear combination of $z_1,\dots,z_t$. Looking at the first $t$ coordinates, the vector $z_i'\in Z_i$ also restricts to the $i$-th standard basis vector, so this linear combination would need to be $z_i'=0\cdot z_1+\dots+0\cdot z_{i-1}+1\cdot z_i+0\cdot z_{i+1}+\dots+0\cdot z_t$. But this means that $z_i=z_i'$, which is a contradiction. Hence $z_i'$ cannot be expressed as a linear combination of $z_1,\dots,z_t$, which implies that $d\ge \dim \spn_{\mathbb{F}_P}(z_1,\dots,z_t,z_i')\ge \dim \spn_{\mathbb{F}_P}(z_1,\dots,z_t)+1= t+1$.
\end{proof}

In order to bound the probability that a given $t$-tuple $(z_1,\dots,z_t)$ is a candidate $t$-tuple but not isolated, we need to sum up the probabilities in Claim \ref{claim-prob-two-tuples-candidates} for all the relevant $t$-tuples $(z_1',\dots,z_t')$ (i.e.\ all $t$-tuples $(z_1',\dots,z_t')\in Z_1\times\dots\times Z_t$ with $z_1'+\dots+z_t'=\mathbbm{1}^{tr}$ and $z_i=z_i'$ for some $i\in \{1,\dots,t\}$). Recall that the relevant probabilities in the sum depend on $d=\dim \spn_{\mathbb{F}_P}(z_1,\dots,z_t,z_1',\dots,z_t')$. So, given $(z_1,\dots,z_t)$, we need to bound the number of such $t$-tuples $(z_1',\dots,z_t')$ for any given value of $d=\dim \spn_{\mathbb{F}_P}(z_1,\dots,z_t,z_1',\dots,z_t')$. Such a bound is stated in the following claim, which is the last missing ingredient in order to prove Lemma \ref{lemma-probability-isolated-candidate}.

\begin{claim}\label{claim-bound-possibilities-z-prime}
Consider a $t$-tuple $(z_1,\dots,z_t)\in Z_1\times\dots\times Z_t$ with $z_1+\dots+z_t=\mathbbm{1}^{tr}$ and an integer $d$ with $t+1\le d\le 2t-1$. Then the number of $t$-tuples $(z_1',\dots,z_t')\in Z_1\times\dots\times Z_t$ with $z_1'+\dots+z_t'=\mathbbm{1}^{tr}$ and $z_i=z_i'$ for some $i\in \{1,\dots,t\}$, which satisfy $\dim \spn_{\mathbb{F}_P}(z_1,\dots,z_t,z_1',\dots,z_t')=d$, is at most
\[t^t\cdot (t-1)!^{(r-1)(d-t)/(t-2)}\]
\end{claim}

For the proof of this claim, we first need a series of further claims as preparations. We will bound the number of $t$-tuples $(z_1',\dots,z_t')\in Z_1\times\dots\times Z_t$ in the claim by looking at all possibilities for the component structure of the auxiliary bipartite graph $G(z_1,\dots,z_t,z_1',\dots,z_t')$ in the following definition.

\begin{definition}\label{def-auxiliary-bipartite-graph}
For $t$-tuples $(z_1,\dots,z_t),(z_1',\dots,z_t')\in Z_1\times\dots\times Z_t$ with $z_1+\dots+z_t=\mathbbm{1}^{tr}$ and $z_1'+\dots+z_t'=\mathbbm{1}^{tr}$, let us define the bipartite graph $G(z_1,\dots,z_t,z_1',\dots,z_t')$ with $t$ vertices on left side (labeled by $1,\dots,t$) and $t$ vertices on the right side (labeled by $1,\dots,t$) as follows. For $(i,j)\in \{1,\dots,t\}^2$ let us draw an edge between vertex $i$ on the left and vertex $j$ on the right if the vectors $z_i\in \{0,1\}^{tr}$ and $z_j'\in \{0,1\}^{tr}$ have a $1$-entry in a common position (i.e.\ if in some position both $z_i$ and $z_j'$ have a one).
\end{definition}

Note for any $(i,j)\in \{1,\dots,t\}^2$, the graph $G(z_1,\dots,z_t,z_1',\dots,z_t')$ has an edge between vertex $i$ on the left and vertex $j$ on the right if and only if the subsets of $\{1,\dots,tr\}$ given by the supports of the vectors $z_i\in \{0,1\}^{tr}$ and $z_j'\in \{0,1\}^{tr}$ have a non-empty intersection.

\begin{claim}\label{claim-number-connected-components}
For any $t$-tuples $(z_1,\dots,z_t),(z_1',\dots,z_t')\in Z_1\times\dots\times Z_t$ with $z_1+\dots+z_t=\mathbbm{1}^{tr}$ and $z_1'+\dots+z_t'=\mathbbm{1}^{tr}$, the graph $G(z_1,\dots,z_t,z_1',\dots,z_t')$ has at least $2t-\dim \spn_{\mathbb{F}_P}(z_1,\dots,z_t,z_1',\dots,z_t')$ connected components.
\end{claim}

We remark that one can show that the number of connected components of $G(z_1,\dots,z_t,z_1',\dots,z_t')$ is actually exactly $2t-\dim \spn_{\mathbb{F}_P}(z_1,\dots,z_t,z_1',\dots,z_t')$ (and not just at least this number), but this won't be relevant for the rest of our argument.

\begin{proof}[Proof of Claim \ref{claim-number-connected-components}]
Let us consider the linear subspace of $\mathbb{F}_P^{2t}$ consisting of all $(a_1,\dots,a_t,a_1',\dots,a_t')\in \mathbb{F}_P^{2t}$ such that $a_1z_1+\dots+a_tz_t-a_1'z_1'-\dots-a_t'z_t'=0$. First, we claim that for any such $(a_1,\dots,a_t,a_1',\dots,a_t')$ we must have $a_i=a_j'$ whenever there is an edge in $G(z_1,\dots,z_t,z_1',\dots,z_t')$ between vertex $i$ on the left and vertex $j$ on the right. Indeed, if there is such an edge, then there exists some $s\in \{1,\dots,tr\}$ such that both the vectors $z_i$ and $z_j'$ have a $1$-entry in the $s$-th position. Since $z_1+\dots+z_t=\mathbbm{1}^{tr}$ and $z_1'+\dots+z_t'=\mathbbm{1}^{tr}$ (as integer vectors), all other vectors among $z_1,\dots,z_t,z_1',\dots,z_t'\in\{0,1\}^{tr}$ have a zero in the $s$-th position. Hence in order for $a_1z_1+\dots+a_tz_t-a_1'z_1'-\dots-a_t'z_t'\in \mathbb{F}_P^{tr}$ to have a zero in the $s$-th position, we must indeed have $a_i=a_j'$ in $\mathbb{F}_P$.

Thus, for every $(a_1,\dots,a_t,a_1',\dots,a_t')\in \mathbb{F}_P^{2t}$ with $a_1z_1+\dots+a_tz_t-a_1'z_1'-\dots-a_t'z_t'=0$ in $\mathbb{F}_P^{tr}$, the entries in $(a_1,\dots,a_t,a_1',\dots,a_t')$ must be constant along every edge of the graph $G(z_1,\dots,z_t,z_1',\dots,z_t')$. Hence the entries in $(a_1,\dots,a_t,a_1',\dots,a_t')$ must also be constant on every connected component of  $G(z_1,\dots,z_t,z_1',\dots,z_t')$. So the dimension of the subspace of all such $(a_1,\dots,a_t,a_1',\dots,a_t')\in \mathbb{F}_P^{2t}$ can be at most the number of connected components of $G(z_1,\dots,z_t,z_1',\dots,z_t')$.

On the other hand, by the rank-nullity theorem, the dimension of the  subspace of all $(a_1,\dots,a_t,a_1',\dots,a_t')\in \mathbb{F}_P^{2t}$ with $a_1z_1+\dots+a_tz_t-a_1'z_1'-\dots-a_t'z_t'=0$ is exactly $2t-\dim \spn_{\mathbb{F}_P}(z_1,\dots,z_t,z_1',\dots,z_t')$. Thus, the number of components of $G(z_1,\dots,z_t,z_1',\dots,z_t')$ is at least  $2t-\dim \spn_{\mathbb{F}_P}(z_1,\dots,z_t,z_1',\dots,z_t')$.
\end{proof}

Note that for all $i=1,\dots,t$ the graph $G(z_1,\dots,z_t,z_1',\dots,z_t')$ in the Definition \ref{def-auxiliary-bipartite-graph} has an edge between vertex $i$ on the left and vertex $i$ on the right (since $z_i,z_i'\in Z_i$ and so both $z_i$ and $z_i'$ have a $1$-entry in the $i$-th position). This means that the component structure of $G(z_1,\dots,z_t,z_1',\dots,z_t')$ can be described by a partition of $\{1,\dots,t\}=I_1\cup\dots\cup I_{\ell}$ into non-empty subsets, where for each subset $I_k$ the vertices in $I_k$ on the left and the vertices in $I_k$ on the right together form a component of $G(z_1,\dots,z_t,z_1',\dots,z_t')$. Here, $\ell$ is the number of components of $G(z_1,\dots,z_t,z_1',\dots,z_t')$, so if $\dim \spn_{\mathbb{F}_P}(z_1,\dots,z_t,z_1',\dots,z_t')=d$, then by Claim \ref{claim-number-connected-components} we must have $\ell\ge 2t-d$. We will prove Claim \ref{claim-bound-possibilities-z-prime} by counting the number of possible $t$-tuples $(z_1',\dots,z_t')\in Z_1\times\dots\times Z_t$ (for some given $(z_1,\dots,z_t)\in Z_1\times\dots\times Z_t$) separately for each of the possible partitions $\{1,\dots,t\}=I_1\cup\dots\cup I_{\ell}$.

\begin{claim}\label{claim-number-z-prime-given-components}
For every given $t$-tuple $(z_1,\dots,z_t)\in Z_1\times\dots\times Z_t$ with $z_1+\dots+z_t=\mathbbm{1}^{tr}$, and every partition $\{1,\dots,t\}=I_1\cup\dots\cup I_{\ell}$, there are at most $(|I_1|!\dotsm |I_\ell|!)^{r-1}$ different $t$-tuples $(z_1',\dots,z_t')\in Z_1\times\dots\times Z_t$ with $z_1'+\dots+z_t'=\mathbbm{1}^{tr}$ such that the component structure of $G(z_1,\dots,z_t,z_1',\dots,z_t')$ is given by the partition $\{1,\dots,t\}=I_1\cup\dots\cup I_{\ell}$.
\end{claim}
\begin{proof}
Recall that any $t$-tuple $(z_1',\dots,z_t')\in Z_1\times\dots\times Z_t$ with $z_1'+\dots+z_t'=\mathbbm{1}^{tr}$ corresponds to a partition of $\{1,\dots,tr\}$ into $t$ sets or size $r$ (given by the supports of the vectors $z_1',\dots,z_t'$), each containing exactly one element from each of the subsets $\{1,\dots,t\}, \{t+1,\dots, 2t\}, \dots, \{t(r-1)+1,\dots, tr\}$. If the components of $G(z_1,\dots,z_t,z_1',\dots,z_t')$ are given by the partition $\{1,\dots,t\}=I_1\cup\dots\cup I_{\ell}$, then for every $k=1,\dots,\ell$ the vertices in $I_k$ on the right side can only have edges to the vertices in $I_k$ on the left side (but not to any of the other vertices on the left side). This means that the supports of the vectors $z_i'$ with $i\in I_k$ can only share elements with the supports of the vectors $z_i$ with $i\in I_k$ (but none of the other vectors among $z_1,\dots,z_t$). So for each of the subsets $\{t+1,\dots, 2t\}, \dots, \{t(r-1)+1,\dots, tr\}$, the $|I_k|$ elements in the union of the supports of the vectors $z_i$ for $i\in I_k$ must be distributed among the supports of the $|I_k|$ vectors $z_i'$ for $i\in I_k$. There are at most $|I_k|!$ possibilities to form such a distribution for each of the subsets $\{t+1,\dots, 2t\}, \dots, \{t(r-1)+1,\dots, tr\}$. In the first subset $\{1,\dots,tr\}$, we already know that every element $i=1,\dots,t$ is contained in the support of $z_i'$ as $(z_1',\dots,z_t')\in  Z_1\times\dots\times Z_t$ (so there is only one possibility for the distribution). Thus, there are at most $(|I_k|!)^{r-1}$ possibilities for choosing the vectors $z_i'$ for $i\in I_k$. In total, this shows that there are at most $(|I_1|!\dotsm |I_\ell|!)^{r-1}$ possibilities for choosing $(z_1',\dots,z_t')\in Z_1\times\dots\times Z_t$ as in the claim statement.
\end{proof}

In order to bound the expression $|I_1|!\dotsm |I_\ell|!$ occurring in Claim \ref{claim-number-z-prime-given-components}, we use the following simple observation.

\begin{claim}\label{claim-factorials}
For integers $1\le a\le b$ with $b\ge 2$, we have $a!\le b!^{(a-1)/(b-1)}$.
\end{claim}
\begin{proof}
For $a=1$, note that $1!=1= b!^{(1-1)/(b-1)}$, and for $a=b$ note that $a!= a!^{(a-1)/(a-1)}$. For $2\le a \le b-1$, we have
\[a!^{b-1}=(2\dotsm a)^{b-1}=(2\dotsm a)^{a-1}\cdot (2\dotsm a)^{b-a}\le (2\dotsm a)^{a-1} \cdot ((a+1)\dotsm b)^{a-1}=b!^{a-1},\]
where the inequality holds because $(2\dotsm a)^{b-a}$ is a product of $(a-1)(b-a)$ factors in $\{2,\dots,a\}$ while $((a+1)\dotsm b)^{a-1}$ is a product of $(a-1)(b-a)$ factors in $\{a+1,\dots,b\}$. Now, rearranging gives the desired inequality.
\end{proof}

Now, we are ready to prove Claim \ref{claim-bound-possibilities-z-prime}.

\begin{proof}[Proof of Claim \ref{claim-bound-possibilities-z-prime}]
For any $t$-tuple $(z_1',\dots,z_t')\in Z_1\times\dots\times Z_t$ with $z_1'+\dots+z_t'=\mathbbm{1}^{tr}$ and $z_i=z_i'$ for some $i\in \{1,\dots,t\}$, such that $\dim \spn_{\mathbb{F}_P}(z_1,\dots,z_t,z_1',\dots,z_t')=d$, let us consider the bipartite graph $G(z_1,\dots,z_t,z_1',\dots,z_t')$ as in Definition \ref{def-auxiliary-bipartite-graph}. The component structure of this auxiliary bipartite graph is given by some partition $\{1,\dots,t\}=I_1\cup\dots\cup I_{\ell}$, and by Claim \ref{claim-number-connected-components} we must have $\ell\ge 2t-d$. Furthermore, at least one of the sets $I_1,\dots, I_{\ell}$ must have size $1$. Indeed, recalling that there is an index $i\in \{1,\dots,t\}$ such that $z_i=z_i'$, one of the sets  $I_1,\dots, I_{\ell}$ must be equal to $\{i\}$. In particular, this means that each of the sets $|I_1|,\dots, |I_{\ell}|$ has size at most $t-1$.

There are at most $t^t$ possible partitions $\{1,\dots,t\}=I_1\cup\dots\cup I_{\ell}$ of $\{1,\dots,t\}$ into at least $2t-d$ non-empty subsets of size at most $t-1$. For each such partition $\{1,\dots,t\}=I_1\cup\dots\cup I_{\ell}$, by Claim \ref{claim-number-z-prime-given-components} there are at most $(|I_1|!\dotsm |I_\ell|!)^{r-1}$ possibilities for $(z_1',\dots,z_t')\in Z_1\times\dots\times Z_t$ with $z_1'+\dots+z_t'=\mathbbm{1}^{tr}$ such that the component structure of $G(z_1,\dots,z_t,z_1',\dots,z_t')$ is given by the partition $\{1,\dots,t\}=I_1\cup\dots\cup I_{\ell}$. Note that by Claim \ref{claim-factorials} (using that $|I_1|,\dots, |I_{\ell}|\le t-1$) we have
\[|I_1|!\dotsm |I_\ell|!\le (t-1)!^{(|I_1|-1)/(t-2)}\dotsm (t-1)!^{(|I_\ell|-1)/(t-2)}=(t-1)!^{(t-\ell)/(t-2)}\le (t-1)!^{(d-t)/(t-2)},\]
where in the last step we used that $\ell\ge 2t-d$. Hence the total number of possibilities for $(z_1',\dots,z_t')\in Z_1\times\dots\times Z_t$ as in Claim \ref{claim-bound-possibilities-z-prime} is at most
\[\sum_{\{1,\dots,t\}=I_1\cup\dots\cup I_{\ell}} (|I_1|!\dotsm |I_\ell|!)^{r-1}\le t^t\cdot \left((t-1)!^{(d-t)/(t-2)}\right)^{r-1}=t^t\cdot (t-1)!^{(r-1)(d-t)/(t-2)},\]
where the sum is over all partitions $\{1,\dots,t\}=I_1\cup\dots\cup I_{\ell}$ of $\{1,\dots,t\}$ into at least $2t-d$ non-empty subsets of size at most $t-1$.
\end{proof}

Finally, we prove Lemma \ref{lemma-probability-isolated-candidate}.

\begin{proof}[Proof of Lemma \ref{lemma-probability-isolated-candidate}]
Recall that $(z_1,\dots,z_t)\in Z_1\times\dots\times Z_t$ is a $t$-tuple with $z_1+\dots+z_t=\mathbbm{1}^{tr}$. We want to show that $(z_1,\dots,z_t)$ is an isolated candidate $t$-tuple with probability at least $R/(2P^{t-1})$. By Claim \ref{claim-probability-single-candidate}, $(z_1,\dots,z_t)$ is a  candidate $t$-tuple with probability  $R/P^{t-1}$. So it remains to show that the probability that $(z_1,\dots,z_t)$ is a candidate $t$-tuple which is not isolated is at most $R/(2P^{t-1})$.

If $(z_1,\dots,z_t)$ is a candidate $t$-tuple, but is not isolated, then there must be another candidate $t$-tuple $(z_1',\dots,z_t')\in Z_1\times \dots\times Z_t$ such that $z_i'=z_i$ for some $i\in \{1,\dots,t\}$. Since $(z_1',\dots,z_t')\neq (z_1',\dots,z_t')$, the dimension $d=\dim \spn_{\mathbb{F}_P}(z_1,\dots,z_t,z_1',\dots,z_t')$ must satisfy $t+1\le d\le 2t-1$ by Claim \ref{claim-values-of-d}. We can therefore bound
\begin{align*}
&\Pr[(z_1,\dots,z_t)\text{ is a candidate $t$-tuple and not isolated}]\\
&\quad\quad\quad\le \sum_{\substack{(z_1',\dots,z_t')\in Z_1\times \dots\times Z_t\\z_1'+\dots+z_t'=\mathbbm{1}^{tr}\\z_i'=z_i\text{ for some }i}}\Pr[(z_1,\dots,z_t)\text{ and }(z_1,\dots,z_t)\text{ are candidate $t$-tuples}]\\
&\quad\quad\quad= \sum_{d=t+1}^{2t-1}\ \sum_{\substack{(z_1',\dots,z_t')\in Z_1\times \dots\times Z_t\\z_1'+\dots+z_t'=\mathbbm{1}^{tr}\\z_i'=z_i\text{ for some }i\\ \dim \spn_{\mathbb{F}_P}(z_1,\dots,z_t,z_1',\dots,z_t')=d}}\Pr[(z_1,\dots,z_t)\text{ and }(z_1,\dots,z_t)\text{ are candidate $t$-tuples}]\\
&\quad\quad\quad\le \sum_{d=t+1}^{2t-1}\ \sum_{\substack{(z_1',\dots,z_t')\in Z_1\times \dots\times Z_t\\z_1'+\dots+z_t'=\mathbbm{1}^{tr}\\z_i'=z_i\text{ for some }i\\ \dim \spn_{\mathbb{F}_P}(z_1,\dots,z_t,z_1',\dots,z_t')=d}}\frac{R}{P^{d-1}}\le \sum_{d=t+1}^{2t-1} t^t\cdot (t-1)!^{(r-1)(d-t)/(t-2)}\cdot \frac{R}{P^{d-1}},
\end{align*}
where in the third step we used Claim \ref{claim-prob-two-tuples-candidates} and in the fourth step Claim \ref{claim-bound-possibilities-z-prime}. Recalling (\ref{eq-def-P}), we have $(t-1)!^{(r-1)(d-t)/(t-2)}\le (P/(2t^{t+1}))^{d-t}\le P^{d-t}/(2t^{t+1})$ for all $d\ge t+1$. Hence
\[\Pr[(z_1,\dots,z_t)\text{ is a candidate $t$-tuple and not isolated}]\le \sum_{d=t+1}^{2t-1} t^t\cdot \frac{P^{d-t}}{2t^{t+1}}\cdot \frac{R}{P^{d-1}}\le \sum_{d=t+1}^{2t-1} \frac{R}{2t\cdot P^{t-1}}\le \frac{R}{2P^{t-1}},\]
as desired.
\end{proof}

We remark that Theorem \ref{thm-fixed-t-F-stronger} can be proved in a very similar way, but the deduction from the known result in Proposition \ref{prop-t-color-sum-free} in the previous subsection is shorter.

\subsection{Proof of the upper bound in Theorem \ref{thm-fixed-t-f}}
\label{subsect-linear-algebra}

As mentioned in the introduction, Munh\'{a} Correia, Sudakov, and Tomon \cite{correia-sudakov-tomon} proved the upper bound $F(t,r)\le (t-1)\cdot {tr\choose r}$ in the setting of Theorem \ref{thm-fixed-t-F}. Their argument is based on linear algebra, and leads to the following linear-algebraic statement.

\begin{theorem}\label{thm-linear-algebra}
Let $V$ be a vector space over a field $\mathbb{F}$, let $t\ge 2$ and let $\varphi: V^t=V\times\dots\times V\to \mathbb{F}$ be a multi-linear map. Consider a collection of $t$-tuples $(x_{j,1},\dots,x_{j,t})\in V^t$ for $j=1,\dots,N$, with $N>(t-1)\cdot \dim V$, such that $\varphi(x_{j,1},\dots,x_{j,t})\ne 0$ for $j=1,\dots,N$. Then there exist distinct indices $j_1,\dots,j_t\in\{1,\dots,N\}$ and elements $y_1,\dots,y_t$ with $y_i\in \{x_{j_i,1},\dots,x_{j_i,t}\}$ for $i=1,\dots,t$ such that $\varphi(y_1,\dots,y_t)\neq 0$.
\end{theorem}

This theorem implies the bound $F(t,r)\le (t-1)\cdot {tr\choose r}$ in Theorem \ref{thm-correia-sudakov-tomon} due to \cite{correia-sudakov-tomon} by taking $V$ to be the $r$-th exterior power of some $(rt)$-dimensional vector space and taking $\varphi(x_1,\dots,x_t) =x_1\wedge \dots\wedge x_t$ (the details of the deduction can be found later in this subsection). 

The question of how large one needs to take $N$ in Theorem \ref{thm-linear-algebra} can be viewed as the natural linear-algebraic abstraction of Problem \ref{problem-F}. Interestingly, while the bound $F(t,r)\le (t-1)\cdot {tr\choose r}$ is not tight in general, the bound $N>(t-1)\cdot \dim V$ in Theorem \ref{thm-linear-algebra} is actually best-possible for any given $t\ge 2$ and any given dimension $\dim V$. Indeed, one can take $V=\mathbb{F}^{d}$ (with standard basis vectors $e_1,\dots, e_d$) for any given $d=\dim V$, take $\varphi: V^t \to \mathbb{F}$ to be the multi-linear defined by $\varphi(x_1,\dots,x_t)=\sum_{s=1}^d x_1^{(s)}\dotsm x_t^{(s)}$ where $x_i=(x_i^{(1)},\dots,x_i^{(d)})$ are the coordinates of $x_i$ for $i=1,\dots,t$, and take collection of $t$-tuples $(x_{j,1},\dots,x_{j,t})\in V^t$ for $j=1,\dots,(t-1)d$ consisting of $t-1$ copies of $(e_s,\dots,e_s)$ for $s=1,\dots,d$.

Munh\'{a} Correia, Sudakov, and Tomon \cite{correia-sudakov-tomon} did not state Theorem \ref{thm-linear-algebra} in this general form (they worked only in the setting where $V$ is an exterior power and $\varphi(x_1,\dots,x_t) =x_1\wedge \dots\wedge x_t$), but their arguments for proving $F(t,r)\le (t-1)\cdot {tr\choose r}$ easily generalize to this setting. They prove and use a result about the max-flattening rank of tensors of a certain diagonal-like shape (called ``semi-diagonal''). For the reader's convenience, we give an alternative self-contained proof of Theorem \ref{thm-linear-algebra}
 
\begin{proof}[Proof of Theorem \ref{thm-linear-algebra}]
Let us prove the theorem by induction on $\dim V$. For $\dim V=0$ the theorem is tautologically true (since $V=\{0\}$ and so it is impossible to have $\varphi(x_{j,1},\dots,x_{j,t})\ne 0$ for any $j$).

Let us now assume that $\dim V\ge 1$. Let us consider all possible choices of indices $j_1,\dots,j_t\in\{1,\dots,N\}$ satisfying the following three conditions:
\begin{itemize}
\item $N$ appears at least once among the indices $j_1,\dots,j_t$.
\item Every element of $\{1,\dots,N-1\}$ appears at most once among the indices $j_1,\dots,j_t$.
\item There exist  elements $y_1,\dots,y_t$ with $y_i\in \{x_{j_i,1},\dots,x_{j_i,t}\}$ for $i=1,\dots,t$ such that $\varphi(y_1,\dots,y_t)\neq 0$.
\end{itemize}
Note that it is indeed possible to choose indices $j_1,\dots,j_t\in\{1,\dots,N\}$ with these three conditions, for example by taking $j_1=\dots=j_t=N$ (then the third condition is satisfied because we can take $y_i=x_{N,i}$ for $i=1,\dots,t$ since $\varphi(x_{N,1},x_{N,2},\dots,x_{N,t})\neq 0$ by the assumption in Theorem \ref{thm-linear-algebra}).

Among all choices of indices $j_1,\dots,j_t\in\{1,\dots,N\}$ satisfying these conditions, let us fix a choice with the minimum possible number of repetitions of $N$. If there is only one repetition of $N$, then $j_1,\dots,j_t$ are distinct and so the conclusion of Theorem \ref{thm-linear-algebra} is satisfied. So let us now assume that among $j_1,\dots,j_t\in\{1,\dots,N\}$ there are $\ell\ge 2$ repetitions of $N$.

By relabeling the indices in $\{1,\dots,N-1\}$ we may assume that besides the $\ell$ repetitions of $N$ the other $t-\ell$ indices among $j_1,\dots j_t$ are $N-1, N-2, \dots, N-t+\ell$. Furthermore, choose an index $i\in \{1,\dots,t\}$ such that $j_i=N$. 

Suppose we had $y_i\in \spn(\bigcup_{j=1}^{N-t+\ell-1}\{x_{j,1},\dots,x_{j,t}\})$. Then we could write $y_i$ as a linear combination of the vectors $x_{j,k}$ for $j=1,\dots,N-t+\ell-1$ and $k=1,\dots,t$. Since $\varphi(y_1,\dots,y_t)\neq 0$, this would mean that we must have $\varphi(y_1,\dots,y_{i-1},x_{j,k},y_{i+1},\dots,y_t)\neq 0$ for some $j\in \{1,\dots,N-t+\ell-1\}$ and $k\in \{1,\dots,t\}$. But then choosing $j_i'=j$ and $y_{i}'=x_{j,k}$, the list of indices $j_1,\dots,j_{i-1},j_i',j_{i+1},\dots,j_t$ obtained from $j_1,\dots,j_t$ upon replacing $j_i$ by $j_i'$ would also satisfy the three conditions above and would have only $\ell-1$ repetitions of $N$. This would  be a contradiction to our choice of $j_1,\dots,j_t$.

Hence $y_i\not\in \spn(\bigcup_{j=1}^{N-t+\ell-1}\{x_{j,1},\dots,x_{j,t}\})$. In particular, this span is not the entire space $V$ and so we can find a subspace $V'\su V$ with $\dim V'=\dim V-1$ such that $\{x_{j,1},\dots,x_{j,t}\}\su V'$ for all $j=1,\dots,N-t+\ell-1$. Note that $N-t+\ell-1\ge N-t+2-1=N-(t-1)>(t-1)\cdot \dim V'$. Hence by the induction hypothesis (applied to $V'$ and the restriction of the multi-linear map $\varphi$ to $V'\times \dots\times V'$), there exist distinct indices $j_1,\dots,j_t\in\{1,\dots,N-t+\ell-1\}$ and elements $y_1,\dots,y_t$ with $j_i\in \{x_{j_i,1},\dots,x_{j_i,t}\}$ for $i=1,\dots,t$ such that $\varphi(y_1,\dots,y_t)\neq 0$. This finishes the induction step.
\end{proof}

Let us now show how Theorem \ref{thm-correia-sudakov-tomon}, due to Munh\'{a} Correia, Sudakov, and Tomon \cite{correia-sudakov-tomon}, follows from Theorem~\ref{thm-linear-algebra} (this argument essentially appears in \cite{correia-sudakov-tomon}, but we repeat it here as motivation for the slightly more complicated deduction of the upper bound in Theorem \ref{thm-fixed-t-f} below).

\begin{proof}[Proof of Theorem \ref{thm-correia-sudakov-tomon} from \cite{correia-sudakov-tomon}]

Assume that $M_1,\dots,M_N$ are matchings as in Problem \ref{problem-F}, with $N>(t-1)\cdot {tr\choose r}$. Assume that the vertices of the underlying hypergraph are labeled by $1,\dots,M$, and for $j=1,\dots,N$, let $e_{j,1},\dots,e_{j,t}\su \{1,\dots,M\}$ be the edges of matching $M_j$. For some infinite field $\mathbb{F}$, consider the $(rt)$-dimensional vector space $\mathbb{F}^{rt}$ and let $V$ be the $r$-the exterior power $\bigwedge^r \mathbb{F}^{rt}$. Let us choose vectors $z_1,\dots,z_M\in \mathbb{F}^{rt}$ in general position, meaning that any $rt$ of these vectors are linearly independent. For any edge $e\su  \{1,\dots,M\}$ of our $r$-uniform hypergraph, let us define $x(e)\in V=\bigwedge^r \mathbb{F}^{rt}$ to be the vector $x(e)= z_{v(1)}\wedge \dots \wedge z_{v(r)}$, where $1\le v(1)<\dots<v(r)\le M$ are the labels of the vertices in $e$. Now, for $j=1,\dots,N$ and $i=1,\dots,t$, define $x_{j,i}=x(e_{i,j})$. Furthermore, let us define $\varphi: V^t \to \mathbb{F}$ to be the multilinear map given by $\varphi(x_1,\dots,x_t)=x_1\wedge\dots\wedge x_t\in \bigwedge^t V\cong \bigwedge^{rt} \mathbb{F}^{rt}\cong \mathbb{F}$ (where, formally, we fix some isomorphism between the $1$-dimensional vector space $\bigwedge^t V$ and $\mathbb{F}$). Since $z_1,\dots,z_M\in \mathbb{F}^{rt}$ are in general position, for any $v(1),\dots, v(tr)\in \{1,\dots,M\}$ we have $z_{v(1)}\wedge \dots\wedge z_{v(tr)}\neq 0$ in $\bigwedge^{rt} \mathbb{F}^{rt}$ if and only if $v(1),\dots, v(tr)$ are distinct. Hence for any edges $e_1,\dots,e_t$ of our hypergraph, we have $\varphi(x(e_1),\dots,x(e_t))\neq 0$ if and only if $e_1,\dots,e_t$ form a matching. In particular, we have $\varphi(x_{j,1},\dots,x_{j,t})=\varphi(x(e_{j,1}),\dots,x(e_{j,t}))\ne 0$ for $j=1,\dots,N$. Since $N>(t-1)\cdot {tr\choose r}=(t-1)\cdot \dim V$, Theorem \ref{thm-linear-algebra} now implies that there exist distinct indices $j_1,\dots,j_t\in \{1,\dots,N\}$ and edges $e_1,\dots,e_t$ with $e_i\in \{e_{j_i,1},\dots,e_{j_i,t}\}=M_j$ for $i=1,\dots,t$ such that for $y_i=x(e_i)$ we have $\varphi(x(e_1),\dots,x(e_t))= \varphi(y_1,\dots,y_t)\neq 0$.  Thus, the edges $e_1,\dots,e_t$ form a rainbow matching.
 \end{proof}

In a similar way, we can deduce the upper bound for the $r$-partite setting for fixed $t$, stated in Theorem \ref{thm-fixed-t-f}, from Theorem \ref{thm-linear-algebra}.

\begin{proof}[Proof of the upper bound in Theorem \ref{thm-fixed-t-f}]
We need to show that for any $N>(t-1)\cdot t^r$, any collection $M_1,\dots,M_N$ of matchings of size $t$ in some $r$-partite $r$-uniform hypergraph $\mathcal{H}$ must have a rainbow matching of size $t$. Let $V_1,\dots, V_r$ be the vertex sets in the $r$-partition of the hypergraph $\mathcal{H}$.  Furthermore, for $j=1,\dots,N$, let $e_{j,1},\dots,e_{j,t}\su \{1,\dots,M\}$ be the edges of matching $M_j$ (each of these edges $e_{j,i}$ contains exactly one vertex from each of the sets $V_1,\dots,V_t$). For some infinite field $\mathbb{F}$, let $W_1,\dots, W_r$ be $t$-dimensional vector spaces over $\mathbb{F}$, let $W=W_1\oplus \dots\oplus W_r$, and let $U$ be the $r$-th exterior power $\bigwedge^r W$. Finally, let $V\su U$ be the subspace of $U$ generated by vectors of the form $w_1\wedge \dots\wedge w_r$ with $w_1\in W_1,\dots,w_r\in W_r$, and note that $\dim V=t^r$.

For $i=1,\dots,t$, let us choose vectors $z(v)\in W_i$ for every $v\in V_i$ in general position, meaning that any $t$ of these vectors in $W_i$ are linearly independent. For any edge $e$ of our $r$-uniform hypergraph, let us define $x(e)\in V\su U=\bigwedge^r W$ to be the vector $x(e)= z(v_1)\wedge \dots \wedge z(v_r)$, where $v_1\in V_1,\dots, v_r\in V_r$ are the vertices in $e$. Now, for $j=1,\dots,N$ and $i=1,\dots,t$, define $x_{j,i}=x(e_{j,i})\in V$. Furthermore, let us define $\varphi: U^t=U\times\dots\times U\to \mathbb{F}$ to be the multilinear map given by $\varphi(x_1,\dots,x_t)=x_1\wedge\dots\wedge x_t\in \bigwedge^t U\cong \bigwedge^{rt} W\cong \mathbb{F}$ (where, formally, we fix some isomorphism between the $1$-dimensional vector space $\bigwedge^t U$ and $\mathbb{F}$). Then the restriction of $\varphi$ to $V^t\su U^t$ is also a multi-linear map. Now, for any $r$-tuples $(v_{1,1},\dots,v_{1,r}),\dots, (v_{t,1},\dots,v_{t,r})\in V_1\times\dots\times V_r$ we have $z(v_{1,1})\wedge \dots\wedge z(v_{1,r})\wedge\dots\wedge z(v_{t,1})\wedge \dots\wedge z(v_{t,r})\neq 0$ in $\bigwedge^t U\cong \bigwedge^{rt} W$ if and only if the $tr$ vertices $v_{1,1},\dots,v_{1,r},\dots, v_{t,1},\dots,v_{t,r}$ are distinct (note that, if $v_{1,1},\dots,v_{1,r},\dots, v_{t,1},\dots,v_{t,r}$ are distinct, then $z(v_{1,i}),\dots, z(v_{t,i})\in W_i$ are linearly independent by our general position assumption) . Hence for any edges $e_1,\dots,e_t$ of $\mathcal{H}$, we have $\varphi(x(e_1),\dots,x(e_t))\neq 0$ if and only if $e_1,\dots,e_t$ form a matching. In particular, we have $\varphi(x_{j,1},\dots,x_{j,t})=\varphi(x(e_{j,1}),\dots,x(e_{j,t}))\ne 0$ for $j=1,\dots,N$. Together with $N>(t-1)\cdot t^r=(t-1)\cdot \dim V$, Theorem \ref{thm-linear-algebra} now implies that there exist distinct indices $j_1,\dots,j_t\in \{1,\dots,N\}$ and elements $e_1,\dots,e_t$ with $e_i\in \{e_{j_i,1},\dots,e_{j_i,t}\}=M_j$ for $i=1,\dots,t$ such that for $y_i=x(e_i)$ we have $\varphi(x(e_1),\dots,x(e_t))= \varphi(y_1,\dots,y_t)\neq 0$.  Thus, the edges $e_1,\dots,e_t$ form a rainbow matching.
\end{proof}

\subsection{Simple explicit lower bound constructions}
\label{subsect-simple-lower-bounds}

The proofs of the lower bounds for $F(r,t)$ and $f(r,t)$ in Theorems \ref{thm-fixed-t-F} and \ref{thm-fixed-t-f} in Sections \ref{subsect-fixed-t-F} and Section \ref{subsect-fixed-t-f} were based on probabilistic arguments. In this subsection, we give some relatively simple explicit lower bound constructions for $F(r,t)$ and $f(r,t)$. If $t$ is fixed and $r$ is large , these lower bounds are slightly weaker than the lower bounds in Theorems \ref{thm-fixed-t-F} and \ref{thm-fixed-t-f}. However, if $t$ and $r$ are roughly equal and both reasonably large, the lower bounds in this subsection are actually better than the bounds obtained from Theorem \ref{thm-construction-fixed-r}, \ref{thm-fixed-t-F} or \ref{thm-fixed-t-f}.

\begin{theorem}\label{thm-F-lower-simple}
For any $t\ge 2$, we have $F(r,t)\ge \binom{(tr-2)/2}{r-1}\cdot 2^{r-1}$
for any even $r\ge 2$, and $F(r,t)\ge \binom{(tr-t-2)/2}{r-2}\cdot 2^{r-2}$ for any odd $r\ge 3$.
\end{theorem}

For fixed matching size $t$ and large uniformity $r$, this lower bound is (up to constant factors depending on $t$) on the order of $(t^t/(t-2)^{t-2})^{r/2}/\sqrt{r}$. In other words, the bound is exponential in $r$ with exponential base  $(t^t/(t-2)^{t-2})^{1/2}$ note that this base is only slightly smaller than the base $t^t/(t-1)^{t-1}$ in Theorem \ref{thm-fixed-t-F} (in fact, the quotient between these two bases converges to $1$ as $t$ grows).

In the $r$-partite setting, our simple lower bound has a similar behavior:

\begin{theorem}\label{thm-f-lower-simple}
For any $t\ge 2$, we have $f(r,t)\ge \sqrt{t(t-1)}^{r-1}$ for any odd $r\ge 3$ and $f(r,t)\ge \sqrt{t(t-1)}^{r-2}$ for any even $r\ge 4$.
\end{theorem}

For fixed $t$ and large $r$, this lower bound is (up to constant factors depending on $t$) equal to $\sqrt{t(t-1)}^{r}$. Here, the exponential base $\sqrt{t(t-1)}$ is only slightly smaller than the exponential base $t$ in Theorem \ref{thm-fixed-t-f}.

We remark that our constructions showing Theorem \ref{thm-F-lower-simple} and Theorem \ref{thm-f-lower-simple} both use hypergraphs with $tr$ vertices. In other words, both constructions actually use perfect matchings. This means that the lower bounds in Theorem \ref{thm-F-lower-simple} and Theorem \ref{thm-f-lower-simple} also apply to the variant of Problem where one requires $M_1,\dots,M_N$ to be perfect matchings (and where one is looking for a rainbow perfect matching).

As mentioned above, our proofs of Theorem \ref{thm-F-lower-simple} and Theorem \ref{thm-f-lower-simple} rely on fairly simple explicit constructions. We start by proving Theorem \ref{thm-F-lower-simple}.

\begin{proof}[Proof of Theorem \ref{thm-F-lower-simple}]
First, assume that $r\ge 2$ is even. Consider $tr$ vertices labeled by $1,\dots,tr-2$, $a$ and $a'$.

Let $\mathcal{X}$ be the collection of all subsets $X\su \{1,\dots,tr-2\}$ of size $r-1$ such that there do not exist two elements $x,x'\in X$ with $x+x'=tr-1$. In other words, $\mathcal{X}$ be the collection of subsets $X\su \{1,\dots,tr-2\}$ with $|X|=r-1$ containing at most one element from each of the $(tr-2)/2$ pairs $\{x,tr-1-x\}$ for $1\le x\le (tr-2)/2$. We have $|\mathcal{X}|=\binom{(tr-2)/2}{r-1}\cdot 2^{r-1}$, since there are $\binom{(tr-2)/2}{r-1}$ possibilities which $r-1$ pairs $X$ should contain an element of, and $2^{r-1}$ possibilities to choose an element from each of these $r-1$ pairs.

We will now define matchings $M_{X}$ of size $t$ for all $X\in \mathcal{X}$, and it is sufficient to show that these matchings do not have a rainbow matching of size $t$.

For any $X\in \mathcal{X}$, consider the two sets $X\cup \{a\}$ and $\{tr-1-x\mid x\in X\}\cup \{a'\}$, which both have size $r$ and are disjoint (by the definition of $\mathcal{X}$). Let these two sets be the first two edges of $M_X$. The remaining vertices outside these two edges are a union of $r(t-2)/2$ pairs of the form $\{x,tr-1-x\}$.  Let us divide these $r(t-2)/2$ pairs into $t-2$  groups of size $r/2$ (in an arbitrary way). For each group the union of the pairs in the group is a set of size $r$, and let this set be an edge of $M_X$. Then in total $M_X$ has $2+(t-2)=t$ edges as desired (and it is easy to see that the $t$ edges of $M_X$ are pairwise disjoint).

Let us now check that the matchings $M_{X}$ for $X\in \mathcal{X}$ do not have a rainbow matching of size $t$. Since there are only $tr$ vertices in total, every matching of size $t$ must be a perfect matching, covering all vertices. So if there was a rainbow matching of size $t$, one of its edges would need to contain $a$, so it would need to be of the form $X\cup \{a\}$ for some $X\in \mathcal{X}$. Now, consider the $r-1$ vertices of the form $tr-1-y$ for $y\in X$ (which are all outside of the edge $X\cup \{a\}$). In the rainbow matching, none of these vertices can be covered by an edge consisting of a union of pairs of the form $\{x,tr-1-x\}$ (indeed, $tr-1-y$ would need to be in the same pair as $y$, but $y$ was already covered by the edge $X\cup \{a\}$). So each of these $r-1$ vertices needs to be covered by an edge containing $a'$. In other words, the rainbow matching must contain the edge $\{tr-1-y\mid y\in X\}\cup \{a'\}$. However, both the edges $X\cup \{a\}$ and $\{tr-1-y\mid y\in X\}\cup \{a'\}$ occur only in the matching $M_X$, so there cannot be a rainbow matching containing both of these edges.

This finishes the proof for even $r\ge 2$. If $r\ge 3$ is odd, we can use the above construction for $r-1$ (which is even) to obtain matchings $M_1,\dots,M_N$ of size $t$ in an $(r-1)$-uniform hypergraph on $t(r-1)$ vertices without a rainbow matching of size $t$, where
\[N=\binom{(t(r-1)-2)/2}{(r-1)-1}\cdot 2^{(r-1)-1}=\binom{(tr-t-2)/2}{r-2}\cdot 2^{r-2}.\]
Let us now take $t$ additional vertices, and for each matching $M_j$ for $j \in \{1,\dots,N\}$, let us add one of these additional vertices to each edge of $M_j$ (such that each of the $t$ additional vertices is added to exactly one of the $t$ edges of $M_j$). This way, we obtain $N$ matchings of size $t$ in an $r$-uniform hypergraph on $tr$ vertices, and there is still no rainbow matching of size $t$.
\end{proof}

Finally, let us prove Theorem \ref{thm-f-lower-simple}, which relies on a similar construction.

\begin{proof}[Proof of Theorem \ref{thm-f-lower-simple}]
This time, let us first  assume that $r\ge 3$ is odd. Let us consider a vertex set of size $tr$, divided into $r$ parts of size $t$. We will construct matchings in the complete $r$-partite $r$-uniform hypergraph $\mathcal{H}$ on this vertex set with this $r$-partition.

Inside each of the first $r-1$ parts of the vertex set, let us label the vertices by $1,\dots,t$. In the last part of the vertex set, let us label the vertices by $a_1,a_2,b_1,\dots,b_{t-2}$. An edge in $\mathcal{H}$ corresponds to an $r$-tuple in $\{1,\dots,t\}^{r-1}\times \{a_1,a_2,b_1,\dots,b_{t-2}\}$.

Now, let $\mathcal{X}\su \{1,\dots,t\}^{r-1}$ be the collection of all $(r-1)$-tuples $(x_1,\dots,x_{r-1})\in \{1,\dots,t\}^{r-1}$ with $x_1\neq x_2, x_3\neq x_4,\dots, x_{r-2}\neq x_{r-1}$ (i.e.\ the collection of all $(r-1)$-tuples $(x_1,\dots,x_{r-1})$ where the first two entries are distinct, the next two entries are distinct, and so on). Note that $|\mathcal{X}|=(t(t-1))^{(r-1)/2}=\sqrt{t(t-1)}^{r-1}$.

Let us now define matchings $M_{(x_1,\dots,x_{r-1})}$ of size $t$ in the hypergraph $\mathcal{H}$ for all $(x_1,\dots,x_{r-1})\in\mathcal{X}$. It will then suffice to show that there is no rainbow matching.

For any $(x_1,\dots,x_{r-1})\in\mathcal{X}$, we begin to form the matching $M_{(x_1,\dots,x_{r-1})}$ by taking the two edges of $\mathcal{H}$ corresponding to the two $r$-tuples $(x_1,\dots,x_{r-1},a_1)$ and $(x_2,x_1,x_4,x_3,\dots,x_{r-1},x_{r-2},a_2)$. Note that these two edges are disjoint (by the definition of $\mathcal{X}$). In each of the $r$ vertex parts there are $t-2$ remaining vertices not covered by these two edges. In the first and the second vertex part, the labels of the uncovered vertices agree (namely, these are the labels in the set $\{1,\dots,t\}\setminus\{x_1,x_2\}$). Similarly, in the third and fourth vertex part, the labels of the uncovered vertices agree, and so on. Hence we can cover the remaining vertices by $t-2$ pairwise disjoint edges corresponding to $r$-tuples of the form $(y_1,\dots,y_{r-1},b_i)\in \{1,\dots,t\}^{r-1}\times \{b_1,\dots,b_{t-2}\}$ with $y_1=y_2, y_3=y_4,\dots,y_{r-2}=y_{r-1}$. Let us add these $t-2$ edges to $M_{(x_1,\dots,x_{r-1})}$, then we indeed obtain a matching of size $t$.

Now suppose for contradiction that there is a rainbow matching of size $t$. Since there are only $tr$ vertices in total, this rainbow matching would need to cover all vertices. Then vertex $a_1$ in the last vertex part is covered by an edge corresponding to an $r$-tuple of the form $(x_1,\dots,x_{r-1},a_1)$ for some $(x_1,\dots,x_{r-1})\in\mathcal{X}$. But now let us look at the vertex in the first vertex part with label $x_2$. This vertex cannot be covered by an edge containing one of the vertices $b_1,\dots,b_{t-2}$, since any such edge must be of the form $(y_1,\dots,y_{r-1},b_i)$ with $y_1=y_2,y_3=y_4,\dots,y_{r-2}=y_{r-1}$ (and so it would also need to contain the vertex with label $x_2$ in the second vertex part, which is already covered by the edge corresponding to $(x_1,\dots,x_{r-1},a_1)$). Hence the edge covering the vertex in the first part with label $x_2$ must contain $a_2$. The same argument can be made for the vertex in the second part with label $x_1$, the vertex in the third part with label $x_4$, the vertex in the fourth part with label $x_3$, and so on (finishing with the vertex in the $(r-1)$-th part with label $x_{r-2}$). In the rainbow matching, each of these $r-1$ vertices must be covered by an edge that also contains $a_2$. In other words, the rainbow matching must contain the edge corresponding to the $r$-tuple  $(x_2,x_1,x_4,x_3,\dots,x_{r-1},x_{r-2},a_2)$, and it also contains the  edge corresponding to the $r$-tuple  $(x_1,\dots,x_{r-1},a_1)$. However, both of these edges only occur in the matching $M_{(x_1,\dots,x_{r-1})}$, so there cannot be a rainbow matching containing both of these edges.

This finishes the proof for odd $r\ge 3$. For even $r\ge 4$, we can take the above construction for $r-1$ (which is odd), which gives $\sqrt{t(t-1)}^{r-2}$ matchings of size $t$ in some $(r-1)$-partite $(r-1)$-uniform hypergraph on $t(r-1)$ vertices. Let us now add an additional vertex part with $t$ vertices and for each of the $\sqrt{t(t-1)}^{r-2}$ matchings, let us add one vertex from the new part to each edge (such that each of the $t$ vertices in the new part is added to exactly one of the $t$ edges of each matching). This way, we obtain  $\sqrt{t(t-1)}^{r-2}$ matchings of size $t$ in an $r$-partite $r$-uniform hypergraph on $tr$ vertices, without a  rainbow matching of size $t$.
\end{proof}

\section{Concluding remarks}

In this paper we study the maximum possible number of matchings of size $t$ (each colored in a different color) in a $r$-uniform hypergraph without a rainbow matching of size $t$. This maximum possible number is denoted by $F(r,t)$. Furthermore, $f(r,t)$ is defined the answer to the analogous problem with the additional assumption that the underlying $r$-uniform hypergraph is $r$-partite. The problem of estimating the functions $F(r,t)$ and $f(r,t)$ has attracted the attention of many different researchers \cite{aharoni-berger,abchs,abkk,aharoni-holzman-jiang,alon,drisko,glebov-sudakov-szabo,correia-sudakov-tomon}. For fixed uniformity $r\ge 3$, we determine $F(r,t)$ and $f(r,t)$ up to constant factors (depending on $r$), showing that both are on the order of $t^r$. This in particular disproves a conjecture of Glebov--Sudakov--Szab\'{o} \cite{glebov-sudakov-szabo}. Furthermore, we determine $F(r,t)$ and $f(r,t)$ up to lower-order terms in the opposite regime, for fixed $t\ge 2$ and large uniformity $r$.

It would be interesting to also understand the behavior of the functions $f(r,t)$ and $F(r,t)$ when both $r$ and $t$ are growing. In Section \ref{subsect-simple-lower-bounds}, we give some additional lower bounds for $f(r,t)$ and $F(r,t)$. These bounds are better than our lower bounds in our main theorems stated in the introduction if $r$ and $t$ are both large and of similar order (if, say, $r$ and $t$ are within constant factors of each other). In particular, if $r=t$, our lower bounds in Section \ref{subsect-simple-lower-bounds} are matching the upper bound $F(r,t)\le (t-1)\binom{tr}{r}$ due to Munh\'{a} Correia--Sudakov--Tomon \cite{correia-sudakov-tomon} and the upper bound $f(r,t)\le (t-1)\cdot t^r$ in Theorem \ref{thm-fixed-t-f} up to polynomial factors in $r=t$.

Finally, we remark that all lower bounds in this paper are based on (partly probabilistic) constructions where $M_1,\dots,M_N$ are matchings of size $t$ in an $r$-uniform hypergraph with $tr$ vertices. In other words, all our lower bounds are based on constructions where $M_1,\dots,M_N$ are \emph{perfect} matchings. It is natural to also ask about Problem \ref{problem-F} (as well as its $r$-partite analogue) in the special setting where $M_1,\dots,M_N$ are perfect matchings, and one is looking for a rainbow perfect matching. Of course, all of the upper bounds for $F(r,t)$ and $f(r,t)$ in particular apply to this special case of perfect matchings. Given the constructions we use in this paper, our lower bounds also all apply to this variant of the problem.

\end{document}